\documentclass[11pt]{article}
\usepackage[margin=1in]{geometry}

\usepackage{graphicx} 
\usepackage{amsfonts}
\usepackage{xcolor}
\usepackage{amsthm}
\usepackage{blindtext}
\usepackage{algorithm,algorithmic}
\usepackage{amssymb,amsmath,natbib,mathtools}
\usepackage{bm}
\usepackage{url}
\usepackage{hyperref}

\DeclareMathOperator{\Mf}{\boldsymbol{\mathscr{M}}}
\DeclareMathOperator{\Pp}{\mathcal{P}}
\DeclareMathOperator{\Ff}{\mathcal{F}}

\DeclareMathOperator{\ww}{\mathbf{w}}

\DeclareMathOperator{\bM}{\mathbf{M}}
\DeclareMathOperator{\bA}{\mathbf{A}}

\DeclareMathOperator{\b1}{\mathbf{1}}

\usepackage[mathscr]{euscript}

\DeclareMathOperator*{\argmax}{argmax}

\newcommand{\ba}{\boldsymbol a}

\newcommand{\bx}{\boldsymbol x}

\newcommand{\bb}{\boldsymbol b} 
\newcommand{\tr}{\mbox{tr}} 

\newcommand{\btheta}{\boldsymbol \theta} 
\newcommand{\bchi}{\boldsymbol \chi} 

\newtheorem{theorem}{Theorem}
\newtheorem{corollary}{Corollary}
\newtheorem{lemma}{Lemma}
\newtheorem{definition}{Definition}

\newcommand{\diag}{\text{ diag} }

\newcommand{\EX}{\mathbb E}

\title{Near-Efficient and Non-Asymptotic Multiway Inference}
\author{
Oscar L\'opez$^1$, Arvind Prasadan$^2$, Carlos~Llosa-Vite$^2$, \\Richard~B.~Lehoucq$^2$, Daniel~M.~Dunlavy$^2$
}
\date{%
    $^1$\textit{Harbor Branch Oceanographic Institute}, Florida Atlantic University, Ft Pierce, FL\\%
    $^2$\textit{Sandia National Laboratories}, Albuquerque, NM and Livermore, CA\\[2ex]%
}
\begin{document}

\maketitle
\begin{abstract} 
We establish non-asymptotic efficiency guarantees for tensor decomposition--based inference in count data models. Under a Poisson framework, we consider two related goals: (i) \emph{parametric inference}, the estimation of the full distributional parameter tensor, and (ii) \emph{multiway analysis}, the recovery of its canonical polyadic (CP) decomposition factors. Our main result shows that in the rank-one setting, a rank-constrained maximum-likelihood estimator achieves multiway analysis with variance matching the Cram\'{e}r--Rao Lower Bound (CRLB) up to absolute constants and logarithmic factors. This provides a general framework for studying ``near-efficient'' multiway estimators in finite-sample settings. For higher ranks, we illustrate that our multiway estimator may not attain the CRLB; nevertheless, CP-based parametric inference remains nearly minimax optimal, with error bounds that improve on prior work by offering more favorable dependence on the CP rank. Numerical experiments corroborate near-efficiency in the rank-one case and highlight the efficiency gap in higher-rank scenarios.
\end{abstract}

\begin{figure}[b!]
        \centering       \includegraphics[width=.9\textwidth]{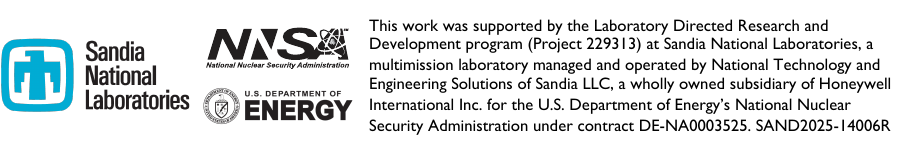}
\end{figure}

\section{Introduction}

Parametric inference and multiway analysis play crucial roles in high-dimensional data analysis. Both perspectives are useful in practice: parametric inference estimates the tensor of distributional parameters as a whole, while multiway analysis yields its latent factors for interpretation \cite{multiwaybook}. Both tasks rely fundamentally on tensor decompositions to represent and exploit underlying structure. However, computing tensor decompositions is notoriously difficult. Degeneracy phenomena lead to non-unique or ill-conditioned factorizations \cite{tensorsuvey} and many tensor problems are NP-hard \cite{NP}, making even approximate computation intractable in general. These issues put into question the reliability of existing tensor-based inference methods. They are particularly pronounced for the canonical polyadic (CP) decomposition \cite{tensorsuvey}, which, despite its widespread use, lacks the theoretical guarantees enjoyed by other tensor formats. Computing CP factors, i.e., multiway analysis, with minimal variance across multiple sets of observations would enhance the reliability of multiway analysis and parametric inference, offering practitioners more confidence in their results while reducing the need for extensive data collection.

To study the reliability of general inference procedures, the Cram\'{e}r--Rao and minimax lower bounds provide fundamental benchmarks for variance-optimal estimators. In Cram\'{e}r--Rao lower bound (CRLB) based analysis \cite{CRLBbook}, the variance of an estimator is bounded from below by a term involving the target parameter's inverse Fisher information matrix. Estimators that achieve the CRLB are known as \emph{efficient} estimators and are particularly valuable because they guarantee optimal performance, such as maximizing prediction stability across experiments \cite{app1,app2,app3,app4}. Similarly, minimax lower bounds capture the best achievable performance of any estimation method for worst-case parameter choices. Notably, minimax analysis establishes worst-case variance guarantees over a class of parameters and all estimators, whereas the CRLB characterizes the smallest achievable variance at a specific parameter value for a given model. Thus, CRLB and minimax analyses offer complementary notions of optimality in structured statistical inference—one local and estimator-specific, the other global and distribution-agnostic—each with its own advantages and limitations.

However, deriving CRLB- and minimax-based guarantees for tensor methods presents significant challenges. On the CRLB side, modern estimators are often defined through constrained or nonconvex optimization procedures, which rarely yield the closed-form expressions required in classical efficiency analysis. Moreover, traditional CRLB results are typically asymptotic—they assume access to infinitely many independent measurements—whereas contemporary data science operates in the finite-sample regime, where only a limited number of observations are available and computational constraints matter. This motivates reformulating efficiency concepts to be meaningful at realistic sample sizes.

On the minimax side, informative lower bounds have been established for tensor estimation problems, and much of the current effort is devoted to matching these bounds with sharp upper bounds. This requires developing new analytical tools that yield sample-complexity guarantees with optimal—i.e., linear—dependence on the tensor rank. However, obtaining such results remains challenging, in large part because it depends on unresolved structural questions surrounding the CP decomposition. Unlike other tensor formats, CP imposes minimal structural assumptions, which makes it widely applicable but also more difficult to analyze rigorously in terms of identifiability and sample complexity.

In this work, we introduce a flexible framework for CRLB-based assessments, establishing several non-asymptotic guarantees for multiway analysis. Focusing on count data modeled by a Poisson distribution, we consider the setting where the array of Poisson rate parameters admits a rank-one representation. In this regime, we show that a rank-one maximum likelihood estimator of the decomposition factors attains an estimation error that matches the CRLB up to absolute constants and logarithmic factors. This result demonstrates that rank-one factor estimators can be \emph{near-efficient}—approaching the theoretical lower variance bound. We propose a relaxed interpretation of the CRLB that is more appropriate for finite-sample settings and enables effective benchmarking of modern estimation methods.

In contrast, the general rank case exhibits significantly different behavior. We argue that estimators for higher-rank CP decompositions generally fail to attain the CRLB as in the rank-one setting. Through numerical experiments, we examine the ``CRLB gap'' and demonstrate that our factor estimation method does not achieve the CRLB. These results reflect a well-known issue of degeneracy in tensor decompositions \cite{tensorsuvey}. Nevertheless, when estimating the full Poisson parameter tensor, we show that CP-based parametric inference remains nearly minimax optimal, up to a rank-dependent factor and a logarithmic term. Our results improve upon previous bounds by providing sharper characterizations of CP rank dependence, thereby reinforcing the utility of decomposition-based inference as a reliable approach for estimating structured distributional parameters.

The remainder of the paper is organized as follows: Subsection \ref{discussion} discusses related work in the literature. In Section \ref{mainsec}, Subsections \ref{CRLBsec} and \ref{minimaxsec} provide background on the CRLB and minimax lower bounds, including a result that we will use to simplify CRLB-based analysis. Afterwards, Section \ref{Manal} will discuss our multiway inference scenario and present a simplified version of our main results along with discussion in Section \ref{maindiscuss}. Numerical experiments are presented in Section \ref{sec:experiments} that explore our theoretical findings. Sections \ref{FIMsec}-\ref{auxlemmas} provide expanded versions of the simplified results presented in Section \ref{mainsec}, along with required lemmas, and proofs. Concluding remarks and future work are discussed in Section 
\ref{conclusion}.

\subsection{Related Work}
\label{discussion}

Our CRLB analysis is motivated by a non-asymptotic perspective, in which efficiency guarantees are established using a finite number of observations. Classical CRLB results—fundamental to statistical inference—are typically asymptotic, assuming the number of samples grows without bound \cite{Akahira1981,279617}. This approach has been especially prevalent in multiway inference, where efficiency is often justified by letting dimension sizes or the number of measurements tend to infinity. However, many applications in modern data science operate in finite-sample regimes, where the number of observations is limited relative to the model dimension. In such settings, estimators may not attain the CRLB \cite{CRLBgap}, e.g., when Fisher information is concentrated in fine-scale model features that cannot be detected from a small number of samples \cite{NEURIPS2022}. These observations do not diminish the value of the classical CRLB; rather, they highlight the need to adapt its principles to finite-sample scenarios. Our work contributes toward this goal by developing a non-asymptotic framework for studying efficiency that preserves the spirit of CRLB analysis while making it applicable to modern high-dimensional tensor models.

Beyond adapting the CRLB to finite-sample settings, a growing body of work has examined statistical inference directly from a non-asymptotic perspective. In \cite{NEURIPS2022}, the authors propose a framework for extending CRLB-style efficiency guarantees to finite samples and show, through concrete examples, that no estimator can achieve the CRLB in general. They address this by smoothing the underlying density and introducing a relaxed notion of efficiency that is attainable with finitely many observations. Other authors consider the accuracy of estimators under finite samples, but not in the context of the CRLB. In \cite{MIAO20101305}, under smoothness conditions on the log-likelihood function, the authors show a concentration inequality of the maximum likelihood estimator from its expectation. The authors in \cite{9719860} consider finite-sample complexity needed to estimate a distribution's mean to a certain accuracy and probability of success. They present an estimator that works for any distribution and matches the optimal sample complexity achieved in the Gaussian case. The work in \cite{27fba59e-c294-3d55-b213-446841721797} is very similar to \cite{9719860}, they show that a wide class of distributions allow for mean estimators that are optimal (near-Gaussian) under finite samples. The work in \cite{Spokoiny2011ParametricEF,Catoni2012} derive concentration properties of maximum likelihood and $M$-estimators, under finite samples and misspecified parametric assumptions.

\section{Main Results}
\label{mainsec}

In this section we present simplified versions of our theoretical results. We begin with background on the CRLB and minimax lower bounds in Sections \ref{CRLBsec} and  \ref{minimaxsec}. Afterwards, Section \ref{Manal} presents our tensor decomposition-based inference problem and results.

\subsection{Cram\'{e}r--Rao Lower Bound}
\label{CRLBsec}

This section presents a multivariate version of the CRLB, the definition of an efficient estimator, and a result that simplifies CRLB-based analysis. To elaborate, suppose $\bx\in\mathbb{R}^n$ is a random vector containing the sample where $f(\bx;\btheta)$ denotes its probability density function (PDF) with $\btheta\in\mathbb{R}^p$ encompassing the distributional parameters we wish to estimate. A simplified version of our CRLB is the following:
\begin{theorem}
\label{CRLBsimp}
Consider an unbiased estimator $\hat\btheta(\bx)$ of $\btheta$ (i.e., $\EX\hat\btheta=\btheta$), where the PDF satisfies the following interchanging of integral and derivative
\begin{equation*}
    \dfrac{\partial}{\partial \btheta} \int_{\mathbb{R}^n} [\hat\btheta(\bchi) - \btheta] f(\bchi;\btheta) d\bchi 
=
\int_{\mathbb{R}^n}\dfrac{\partial}{\partial \btheta}  \left([\hat\btheta(\bchi) - \btheta] f(\bchi;\btheta)\right) d\bchi .
\end{equation*}
Define the Fisher information matrix (FIM) as
$$
\mathcal{I}_{\btheta} \coloneqq \EX\bigg[\left(\frac{\partial}{\partial \btheta}\log f(\bx;\btheta)\right)\left(\frac{\partial}{\partial \btheta}\log f(\bx;\btheta)\right)^{\top}\bigg].
$$
Then
$$
\mathcal{I}_{\btheta}^{\dagger}\preceq \emph{Var}\left(\hat\btheta\right),
$$
where 
$$
\emph{Var}\left(\hat\btheta\right) \coloneqq \EX\Big[\left(\hat\btheta - \btheta\right)\left(\hat\btheta - \btheta\right)^{\top}\Big]
$$ 
is the covariance matrix, $A\preceq B$ means that $B-A$ is positive semi-definite (PSD), and the $\dagger$ superscript denotes the pseudo-inverse.
\end{theorem}

Theorem \ref{CRLBsimp} follows from a more general result that additionally considers biased estimators. For ease of exposition, we focus on this CRLB for now but refer the reader to Section \ref{CRLBproof} for the general result and its proof. This fundamental result provides a manner to gauge the variability of an estimator. The approach considers a PSD-based hierarchy and labels an estimator as \emph{efficient} if it achieves the ``lowest'' possible covariance:

\begin{definition} \label{eff-est}
Under the conditions of Theorem \ref{CRLBsimp}, $\hat\btheta$ is said to be an \textbf{efficient} estimator if
\begin{equation}
\label{efficient}
    \mathcal{I}_{\btheta}^{\dagger}= \emph{Var}\left(\hat\btheta\right).
\end{equation}
\end{definition}

Establishing \eqref{efficient} for an estimator is a major concern in statistical inference, to assure an estimation procedure is optimal in the sense of error variance. This  is a challenging task in multivariate problems where computations of the FIM and its pseudo-inverse (FIM-PI) are cumbersome. Furthermore, many modern inference procedures do not provide a straightforward manner to compute an estimator in closed form, which makes it difficult to compute the covariance matrix. Finally, as noted in Section \ref{discussion}, \eqref{efficient} may not hold in non-asymptotic scenarios \cite{NEURIPS2022,CRLBgap}. This last observation stresses that CRLB-based analysis may not be adequate for finite-sample settings, which are most informative in practice.

To alleviate non-asymptotic CRLB-based analysis, our following result simplifies the establishment of estimator efficiency:
\begin{theorem}
\label{CRLBdist}
Under the conditions of Theorem \ref{CRLBsimp}, $\hat\btheta$ is efficient if and only if
\begin{equation}
\label{dist}
\emph{tr}\left(\mathcal{I}_{\btheta}^{\dagger}\right) = \mathbb{E}\big\|\hat\btheta-\btheta\big\|_2^2,
\end{equation}
where \emph{tr}$(\circ)$ is the matrix trace.
\end{theorem}
The proof is straightforward, we quickly present it before continuing.

\begin{proof}[Proof of Theorem \ref{CRLBdist}]
Since Theorem \ref{CRLBsimp} follows, we have that $\mbox{Var}\left(\hat\btheta\right)-\mathcal{I}_{\btheta}^{\dagger}$ is a PSD matrix. We now apply a well known trace lower bound for PSD matrices (e.g., see Section 2 in \cite{MCI}):
$$
\mbox{if} \ 0\preceq A, \ \mbox{then} \ \|A\| \leq \tr(A),
$$
where $\|\circ\|$ is the operator norm. Therefore
\[
\Big\| \mbox{Var}\left(\hat\btheta\right)-\mathcal{I}_{\btheta}^{\dagger}\Big\| \leq \tr\left( \mbox{Var}\left(\hat\btheta\right)-\mathcal{I}_{\btheta}^{\dagger}\right) =  \mathbb{E}\big\|\hat\btheta-\btheta\big\|_2^2 - \tr\left(\mathcal{I}_{\btheta}^{\dagger}\right),
\]
and the result follows.
\end{proof}

While simple, Theorem \ref{CRLBdist} seems to be unnoticed in the literature. This result severely simplifies the study of efficiency since we must now only consider scalar terms, the FIM-PI trace and the mean squared error $\mathbb{E}\|\hat\btheta-\btheta\|_2^2$ (MSE). The main advantage is that the MSE is well studied for outputs of optimization programs under finite-samples, even when closed form solutions are not available. We propose to use \eqref{dist} to gauge how close an estimator is to achieving the CRLB. Our main argument is that this approach alleviates non-asymptotic CRLB-based analysis. To this end, we introduce the more flexible concept of \emph{near-efficient} estimators:

\begin{definition} \label{neff-est}
Given a random observation $\textbf{x}\sim f(\btheta)$, let $\widetilde{\btheta}$ be an estimator of $\btheta$. We say that $\widetilde{\btheta}$ is a \textbf{near-efficient} estimator if the terms
\[
\emph{tr}\left(\mathcal{I}_{\btheta}^{\dagger}\right) \ \ \mbox{and} \ \ \emph{Var}\left(\widetilde{\btheta}\right)
\]
match up to absolute constants and logarithmic terms.
\end{definition}

Notice that Definition \ref{neff-est} does not require $\widetilde{\btheta}$ to be unbiased, which differs from the unbiasedness assumption in Theorem \ref{CRLBsimp}. Intuitively, demonstrating that an estimator is near-efficient under this definition shows that its variance remains comparable to that of the best unbiased estimator. Although this introduces a mild mismatch in assumptions, the concept remains highly informative in practice, as it enables CRLB-based assessments in realistic finite-sample settings where unbiasedness is rarely attainable. Our approach will consist of deriving error bounds to compare the MSE to the FIM-PI trace, to show that Definition \ref{neff-est} is satisfied. This will be particularly useful in multiway analysis, where only asymptotic results seem to be available in the literature. 

\subsection{Minimax Lower Bounds}
\label{minimaxsec}

We now provide background on minimax lower bounds and explain their relation to the CRLB framework discussed previously. Both approaches characterize fundamental limits of statistical estimation, but they differ in perspective and scope.

Minimax theory studies the best possible MSE that any estimator can achieve over a class of parameters $S$. Formally, one considers
\begin{equation}
    \inf_{\hat{\btheta}}
    \;\sup_{\btheta \in S}
    \mathbb{E}\ \|\hat{\btheta} - \btheta\|_2^2.
    \label{eq:minimax-risk}
\end{equation}
A lower bound on \eqref{eq:minimax-risk} implies that no estimator can uniformly attain smaller MSE over the entire class $S$. Such bounds typically depend on the sample size, the ambient dimension, constraints such as sparsity or tensor rank, or properties of the noise model. Importantly, minimax bounds are \emph{information-theoretic}: they apply to all estimators, not just a particular algorithm.

Both minimax theory and the CRLB quantify optimality, but in complementary ways:
\begin{itemize}
    \item The CRLB is a \emph{local} bound: it applies to a specific parameter $\btheta$ and typically requires computing the FIM and the covariance of a chosen estimator.
    \item Minimax bounds are \emph{global}: they measure the best achievable performance uniformly over a class $S$ and do not depend on any particular estimator.
    \item The CRLB yields a matrix inequality, whereas minimax bounds typically yield scalar lower bounds such as \eqref{eq:minimax-risk}.
    \item Classical CRLB results are primarily asymptotic, whereas minimax theory often provides non-asymptotic error rates.
\end{itemize}
Thus, while the CRLB characterizes the smallest variance achievable by an estimator at a fixed parameter, minimax bounds characterize the smallest possible worst-case error across all estimators and all parameters in $S$.

The choice between CRLB and minimax analysis depends on the application. If the goal is to evaluate a specific estimator and quantify how close it is to an information-theoretic limit at a fixed $\btheta$, the CRLB is the appropriate tool. However, CRLB analysis often requires detailed knowledge of the estimator and may become analytically intractable for large, structured models. Minimax bounds, in contrast, provide broad performance guarantees and are easier to generalize, although they may offer less detailed insight at specific parameter points. In this sense, CRLB and minimax approaches should be viewed as complementary rather than competing frameworks for understanding fundamental limits of structured statistical inference.

\subsection{Near-Efficient Multiway Inference}
\label{Manal}

We now describe our multiway inference problem, estimation method, and main results in the rank-one and general CP rank case. We consider $N$-way Poisson observations from an entry-wise independent parameter model:
\[
\boldsymbol{\mathscr{X}} \sim \mbox{Poisson}\left( \boldsymbol{\mathscr{M}} \right).
\]
Here, $\boldsymbol{\mathscr{X}}\in\mathbb{N}^{I\times\cdots\times I}$ and $\boldsymbol{\mathscr{M}}\in\mathbb{R}_+^{I\times\cdots\times I}$ are $N$-way tensors where $\mathbb{R}_+$ denotes the set of positive real numbers. The entry of $\boldsymbol{\mathscr{X}}$ at the multi-index $\textbf{i}$ is an independent random variable (denoted as $x_{\textbf{i}}$) such that for any $k\in\mathbb{N}$
\[
\mathbb{P}\left(x_{\textbf{i}} = k\right) = \frac{m_{\textbf{i}}^ke^{-m_{\textbf{i}}}}{k!},
\]
where $m_{\textbf{i}}$ denotes the respective entry of $\boldsymbol{\mathscr{M}}$. The array $\boldsymbol{\mathscr{X}}$ consists of observed count data and $\boldsymbol{\mathscr{M}}$ is the tensor of Poisson parameters we wish to estimate, either as an array (parametric inference) or in decomposed form (multiway analysis).

Before proceeding, it is useful to distinguish these two perspectives. Parametric inference refers to estimating the parameter tensor $\boldsymbol{\mathscr{M}}$, i.e., the $I \times \cdots \times I$ independent Poisson parameters, a well-defined statistical task for which minimax lower bounds provide sharp benchmarks. In contrast, multiway analysis aims to recover the canonical polyadic (CP) decomposition factors of $\boldsymbol{\mathscr{M}}$, which is typically much more challenging and may be ill-posed in larger-rank settings due to degeneracy and identifiability issues. Our results highlight this distinction: in the rank-one case, multiway analysis achieves near-efficiency relative to the CRLB, while in the general rank case, CP factor estimation falls short of the CRLB even though parametric inference of $\boldsymbol{\mathscr{M}}$ remains nearly minimax optimal. 

\subsubsection{Rank-one Case: Near-Efficient Multiway Analysis}
We first assume $\boldsymbol{\mathscr{M}}$ is a rank-one tensor
\[
\boldsymbol{\mathscr{M}} = \textbf{u}^{(1)}\circ\cdots \circ \textbf{u}^{(N)},
\]
where $\{\textbf{u}^{(n)}\}_{n=1}^{N}\subset\mathbb{R}^{I}$ are the factors and $\circ$ denotes the vector outer product. Entry-wise, this means that at the multi-index $\textbf{i} = (i_1,\cdots,i_N)$
\[
m_{\textbf{i}} = u_{i_1}^{(1)}u_{i_2}^{(2)}\cdots u_{i_N}^{(N)}
\]
where $u_{i_n}^{(n)}$ is the respective entry of $\textbf{u}^{(n)}$. We remark that most tensor decomposition choices agree in the rank-one case so the result of this section applies beyond the CP decomposition.

Given $\boldsymbol{\mathscr{X}}$, our goal is to estimate factors of the Poisson parameter tensor $\{\textbf{u}^{(n)}\}_{n=1}^{N}$, i.e., multiway analysis. Throughout, we denote $\boldsymbol{\theta}\in \mathbb{R}^{NI}$ as the concatenation of the $\textbf{u}^{(n)}$'s. We will produce an estimator of $\boldsymbol{\theta}$ whose variance roughly matches the respective CRLB. We will consider an estimator that exploits the rank-one Poisson model and that the factor entries lie in some interval $[\beta,\alpha]$ with $\beta>0$, by searching for estimators in
\begin{equation}
\label{S1}
    \mathcal{S}_1(\beta,\alpha) \coloneqq \Big\{ \boldsymbol{\mathscr{T}} = \textbf{w}^{(1)}\circ\cdots \circ \textbf{w}^{(N)} \ | \ \{\textbf{w}^{(n)}\}_{n=1}^{N} \subset [\beta,\alpha]^I \Big\}.
\end{equation}
Our approach will output an estimator of $\boldsymbol{\mathscr{M}}$ as
\begin{equation}
\label{maxlike}
    \widetilde{\boldsymbol{\mathscr{M}}}_1 \coloneqq \argmax_{\boldsymbol{\mathscr{T}}\in\mathcal{S}_1}\mathcal{L}_{\boldsymbol{\mathscr{X}}}(\boldsymbol{\mathscr{T}}) \coloneqq \argmax_{\boldsymbol{\mathscr{T}}\in\mathcal{S}_1} \sum_{\textbf{i}}(x_{\textbf{i}}+1)\log\left(t_{\textbf{i}} + 1\right) - (t_{\textbf{i}} + 1),
\end{equation}
where we write $\mathcal{S}_1(\beta,\alpha)$ as $\mathcal{S}_1$ for compactness. This method is very similar to standard Poisson max-likelihood \cite{PoissonMC,ZTP}, but differs slightly due the the ``$+1$'' terms that appear in the objective. We refer to this estimation method as \emph{shifted-Poisson} regression. The $+1$ shifts advantageously avoid zero counts in order to simplify our analysis. 

The estimator $\widetilde{\boldsymbol{\mathscr{M}}}_1$ is output in rank-one format, with factors $\{\widetilde{\textbf{u}}^{(n)}\}_{n=1}^{N}$ and we let $\widetilde{\btheta}$ denote the estimate factors concatenated as a vector. We now present a simplified theorem that summarizes our rank-one results. The following statement treats terms that do not involve the size of the ambient dimensions, $I$, as constants for ease of exposition.

\begin{theorem}[rank-one multiway analysis]
\label{simpthm}
    Let $\beta >0$ and assume that $\boldsymbol{\mathscr{M}}\in\mathcal{S}_1$. Under our rank-one Poisson observation setting, the FIM-PI satisfies
    \begin{equation*}
\emph{tr}\left(\mathcal{I}_{\btheta}^{\dagger}\right) = \mathcal{O}\left(\frac{1}{I^{N-2}}\right).
    \end{equation*}
    Furthermore, consider the factor estimation procedure based on \eqref{maxlike} with the factors of $\boldsymbol{\mathscr{M}}$ and $\widetilde{\boldsymbol{\mathscr{M}}}_1$ scaled such that $\|\textbf{u}^{(1)}\|_2 = \cdots =\|\textbf{u}^{(N)}\|_2$ and $\|\widetilde{\textbf{u}}^{(1)}\|_2 = \cdots =\|\widetilde{\textbf{u}}^{(N)}\|_2$. If $I$ is large enough (see \eqref{Isize}), the factor estimator obeys
    \begin{equation*}
    \mathbb{E}\big\|\widetilde{\btheta}-\btheta\big\|_2^2 = \mathcal{O}\left(\frac{\log^2(I)}{I^{N-2}}\right).
    \end{equation*}
\end{theorem}

The proof combines Theorem \ref{trFIM} and Corollary \ref{ZTPboundfac} that provide explicit bounds for tr$\left(\mathcal{I}_{\btheta}^{\dagger}\right)$ and the MSE respectively. We refer the reader to Sections \ref{FIMsec} and \ref{MSEsec} for the full statements and proofs. Theorem \ref{simpthm} states that up to absolute constants and logarithmic terms, the trace of the FIM-PI and the MSE of our estimator match. According to Definition \ref{neff-est}, we have shown that our estimator is near-efficient. The result shows that our method is nearly as efficient as the best possible unbiased estimation method. However, as a corollary, we can establish asymptotic efficiency and unbiasedness.

\begin{corollary}
Under the setting of Theorem \ref{simpthm}, $\widetilde{\btheta}$ is asymptotically unbiased and efficient.
\end{corollary}

The corollary follows by allowing $I$ or $N$ to be arbitrarily large, which shows that both tr$\left(\mathcal{I}_{\btheta}^{\dagger}\right)$ and $\mathbb{E}\big\|\widetilde{\btheta}-\btheta\big\|_2^2$ approach zero. In the limit, we must have $\mathbb{E}\widetilde{\btheta} = \btheta$ and both terms match so that Theorem \ref{CRLBdist} is satisfied. This corollary demonstrates the flexibility of our approach, which further allows us to establish asymptotic results.

\subsubsection{Larger Rank Case: Near-Minimax Optimal Parametric Inference}
We further consider the general case where $\boldsymbol{\mathscr{M}}$ is allowed to have rank larger than one. In contrast to the rank-one case, the larger rank scenario presents many complications when considering the CP decomposition. In this setting, our analysis cannot guarantee near-efficient CP factor estimation. Section \ref{sec:experiments} will explore this situation numerically, illustrating that rank-$2$ and rank-$5$ CP decomposition estimation may not achieve the CRLB. 

However, Poisson parameter inference can be shown to be near minimax optimal via our shifted estimator. In other words, we now consider estimating the entire array $\boldsymbol{\mathscr{M}}$ rather than targeting its CP factors as in the rank-one case. We will provide an estimator $\widetilde{\boldsymbol{\mathscr{M}}}$ and study its accuracy with an MSE bound, which we then gauge in optimality via minimax analysis rather than the CRLB.

To elaborate, in the rank-$R$ scenario, we now consider the search space
\[
\mathcal{S}_R(\tilde{\beta},\tilde{\alpha}) \coloneqq \Bigg\{ \boldsymbol{\mathscr{T}} = \sum_{r=1}^{R}\textbf{w}_r^{(1)}\circ\cdots \circ \textbf{w}_r^{(N)} \ \Big| \ \textbf{w}_r^{(n)}\in \mathbb{R}^I, \ \ \tilde{\beta}\leq t_{\textbf{i}}\leq\tilde{\alpha} \Bigg\},
\]
and produce the estimator
\begin{equation}
\label{maxlikeR}
    \widetilde{\boldsymbol{\mathscr{M}}} \coloneqq \argmax_{\boldsymbol{\mathscr{T}}\in\mathcal{S}_R}\mathcal{L}_{\boldsymbol{\mathscr{X}}}(\boldsymbol{\mathscr{T}}),
\end{equation}
where we write $\mathcal{S}_R(\tilde{\beta},\tilde{\alpha})$ as $\mathcal{S}_R$ for compactness. Notice that \eqref{maxlikeR} extends \eqref{maxlike} to search for rank-$R$ tensors, but imposes bounded Poisson parameters rather than factor entry bounds. Our main result in this setting is the following:

\begin{theorem}[Rank-$R$ parametric inference]
\label{MSEthm}
    Let $\tilde{\beta}>0$, assume that $\boldsymbol{\mathscr{M}}\in\mathcal{S}_R(\tilde{\beta},\tilde{\alpha})$ and let $\widetilde{\boldsymbol{\mathscr{M}}}$ be given as in \eqref{maxlikeR}. There exists absolute constants $C_1,C_2>0$ such that for any $\delta>0$ we have
    \begin{equation}
\label{msebound}
    \big\|\widetilde{\boldsymbol{\mathscr{M}}}-\boldsymbol{\mathscr{M}}\big\|_F^2 \leq \frac{C_2(\tilde{\alpha}+1)^4(N+\delta)^2\log^2(I)NIR^2}{\left(\tilde{\beta}+1\right)^2}
\end{equation}
with probability exceeding $1-I^{-\delta}-\exp(-C_1NI\log(N))$.
\end{theorem}

In contrast to the rank-one case, which bounds the CP factors $\btheta$, the rank greater than one case bounds the Poisson parameters. To gauge the optimality of this result we deviate from CRLB analysis and instead present a minimax lower bound.
\begin{theorem}[Minimax lower bound]
\label{Mini}
    Under the Poisson observation model, $\boldsymbol{\mathscr{X}}\sim$ Poisson$\left(\boldsymbol{\mathscr{M}}\right)$, let $\inf_{\widehat{\boldsymbol{\mathscr{M}}}}$ denote the infimum over all estimators $\widehat{\boldsymbol{\mathscr{M}}}\in \mathcal{S}_R(\tilde{\beta},\tilde{\alpha})$ based on the count tensor observation $\boldsymbol{\mathscr{X}}$. 
    Assume that $ R \leq I$, $IR > 16$, and that
    \begin{equation}
        \label{condmini}
        \left(\frac{IR}{ 16} - 1\right) \frac{\tilde{\beta}\log 2}{(\tilde{\alpha} - \tilde{\beta})^2} \leq I^N.
    \end{equation}
    Then,
    \begin{equation}
    \label{minibound}
        \inf_{\widehat{\boldsymbol{\mathscr{M}}}}\sup_{\boldsymbol{\mathscr{M}}\in S_R}\EX\big\|\widehat{\boldsymbol{\mathscr{M}}}-\boldsymbol{\mathscr{M}}\big\|_F^2 \geq \frac{\tilde{\beta}\log 2}{128} \left(\frac{IR}{16} - 1 \right).
    \end{equation}
\end{theorem}
The bound \eqref{minibound} illustrates the best derivable MSE rate for parameter tensors in $\mathcal{S}_R$. The result shows that ``worst-case'' Poisson parameters (achieving the supremum) will suffer an MSE of the order $IR$, and therefore one cannot expect to improve upon this rate for general tensors. Comparing \eqref{minibound} to the MSE upper bound \eqref{msebound}, $\sim IR^2\log^2(I)$, we see that our shifted-Poisson estimator is optimal up to a rank factor and a logarithmic term. Theorem \ref{MSEthm} improves on previous CP-based bounds from the literature in terms of rank dependence; see Section \ref{maindiscuss} for further discussion.

\subsection{Discussion}
\label{maindiscuss}

As in \cite{NEURIPS2022,CRLBgap}, our framework illustrates an interesting trade-off where non-asymptotic analysis comes at the loss of strict efficiency. Our approach introduces the more flexible concept of near-efficiency which can be achieved under limited observations. In contrast to other work, the methodology provides a simplified approach to CRLB-based analysis under finite samples while not deviating from the original definition of efficiency. The approach simplifies the study of estimators derived from constrained optimization programs, which typically do not involve straightforward and closed-form solutions and its utility goes beyond multiway analysis.

In the literature of tensor decompositions, our work is novel in its attempt to express error bounds in terms of the CP factors $\btheta$ (muliway analysis, as in Theorem \ref{simpthm}). In contrast, most results focus on error bounds involving the parameter tensor $\boldsymbol{\mathscr{M}}$ (i.e., parametric inference). In this context, Theorem \ref{MSEthm} improves upon the literature in terms of CP rank dependency. To the best of our knowledge, all results in the literature exhibit error terms with polynomial dependence on the CP rank or largest array in terms of $N$ (e.g., \cite{minimaxbin} with MSE $\sim IR^{N-1}$). This gap is also apparent in cases with missing information \cite{montanari,incoherence,navid,navid2,navid3,harris2023spectral} (e.g., tensor completion). In contrast, Theorem \ref{MSEthm} derives the rate MSE $\sim IR^2\log^2(I)$ which removes the exponential dependence on the order. 

We note that work in the literature has been able to provide optimal sample complexity results for other factorization choices: Tucker decomposition \cite{Anru1,Anru2}, higher order singular
value decomposition, tensor train, and hierarchical Tucker format \cite{IHT}. This highlights that the CP decomposition lacks theoretical understanding and CP-based analysis is more challenging than alternative multiway array decomposition choices. One of our contributions is to improve nuclear norm-based bounds in terms of the CP rank (see Theorem \ref{rankRthm}). This leads to our improved error rates and may be of interest on its own.

\section{Numerical Experiments}
\label{sec:experiments}

This section presents numerical experiments that validate Theorem \ref{simpthm} in the rank-one case and illustrate why its conclusions do not extend to rank larger than one settings. All experiments were implemented in MATLAB using the Tensor Toolbox \cite{Bader2021TensorToolbox} and its generalized CP (GCP) optimization framework. The purpose of these experiments is to empirically evaluate the correspondence between the mean squared error (MSE) and the trace of the FIM-PI, tr$\left(\mathcal{I}_{\btheta}^{\dagger}\right)$, under the shifted-Poisson regression model to achieve multiway analysis.

\subsection{Rank-one Case}

We construct a Poisson parameter tensor $\boldsymbol{\mathscr{M}}$ of order $N$ and dimensions $I$ with entries constrained to the interval $[\beta,\alpha]$. The tensor is formed by taking the outer product of $N$ nonnegative factor vectors, each generated randomly with entries uniformly distributed in the interval $[\beta^{1/N}, \alpha^{1/N}]$. Thus, $\boldsymbol{\mathscr{M}} = \mathbf{u}^{(1)} \circ \cdots \circ \mathbf{u}^{(N)}$, ensuring that all tensor entries lie within the prescribed bounds.

Given $\boldsymbol{\mathscr{M}}$, we generate Poisson-distributed observations $\boldsymbol{\mathscr{Y}} \sim \operatorname{Poisson}(\boldsymbol{\mathscr{M}})$ using \texttt{poissrnd},
and estimate the tensor via shifted-Poisson regression:\begin{equation}
\label{maxlike2}
    \widetilde{\boldsymbol{\mathscr{M}}} \coloneqq \argmax_{\boldsymbol{\mathscr{T}}\in\widetilde{\mathcal{S}}_1} \sum_{\textbf{i}}(x_{\textbf{i}}+1)\log\left(t_{\textbf{i}} + 1\right) - (t_{\textbf{i}} + 1).
\end{equation}
where 
\[
\widetilde{\mathcal{S}}_1 \coloneqq \Big\{ \boldsymbol{\mathscr{T}} = \textbf{w}^{(1)}\circ\cdots \circ \textbf{w}^{(N)} \ | \ \{\textbf{w}^{(n)}\}_{n=1}^{N} \subset (0,\infty)^I \Big\}.
\]
This program is very similar to \eqref{maxlike}, but only imposes a non-negativity constraint on the factor entries. We solve \eqref{maxlike2} using the \verb|gcp_opt| solver, which outputs our estimator in rank-one format $\widetilde{\btheta} = \left(\widetilde{\textbf{u}}^{(1)},\cdots,\widetilde{\textbf{u}}^{(N)}\right)$. The optimization is performed using the \texttt{lbfgsb} quasi-Newton solver with nonnegativity constraints. We scale the factors such that $\|\textbf{u}^{(1)}\|_2 = \cdots =\|\textbf{u}^{(N)}\|_2$ and $\|\widetilde{\textbf{u}}^{(1)}\|_2 = \cdots =\|\widetilde{\textbf{u}}^{(N)}\|_2$, and then compute the error $\|\btheta-\widetilde{\btheta}\|_2^2$. 

To compute tr$\left(\mathcal{I}_{\btheta}^{\dagger}\right)$, we use Lemma 4 in \cite{PCPFIMpaper} that provides the structure of our Poisson Fisher Information matrix. The true factors of $\boldsymbol{\mathscr{M}}$ are further normalized to have equal $\ell_1$-norm (i.e., equal sums across modes). This normalization is to agree with the proof of Theorem 6 that simplifies calculations, but we note that this choice does not seem to severely affect the value of tr$\left(\mathcal{I}_{\btheta}^{\dagger}\right)$ in contrast to other normalization options.

The quantities $\mathrm{MSE}$ and tr$\left(\mathcal{I}_{\btheta}^{\dagger}\right)$ are averaged over 50 realizations of random parameter tensors with $\beta=1$ and $\alpha = 2$. Each experiment is repeated for ambient dimension sizes $I=50$, $75$, and $100$, and for tensor orders $N=3$ and $N=4$. Figure \ref{plot1} reports the mean value with error bars corresponding to the minimum and maximum observed values across repetitions.

Figures \ref{plot1} demonstrates that Theorem \ref{simpthm} holds in practice. In all scenarios, we observe near-efficiency since the estimated MSE and FIM-PI trace match. Furthermore, the order of these terms is in agreement with our main result as $I$ and $N$ vary. We note that the Poisson parameter bounds ($\beta,\alpha$) do not seem to have a significant influence on the behavior of the plots.

\begin{figure}
    \centering
    \includegraphics[width=0.4\linewidth]{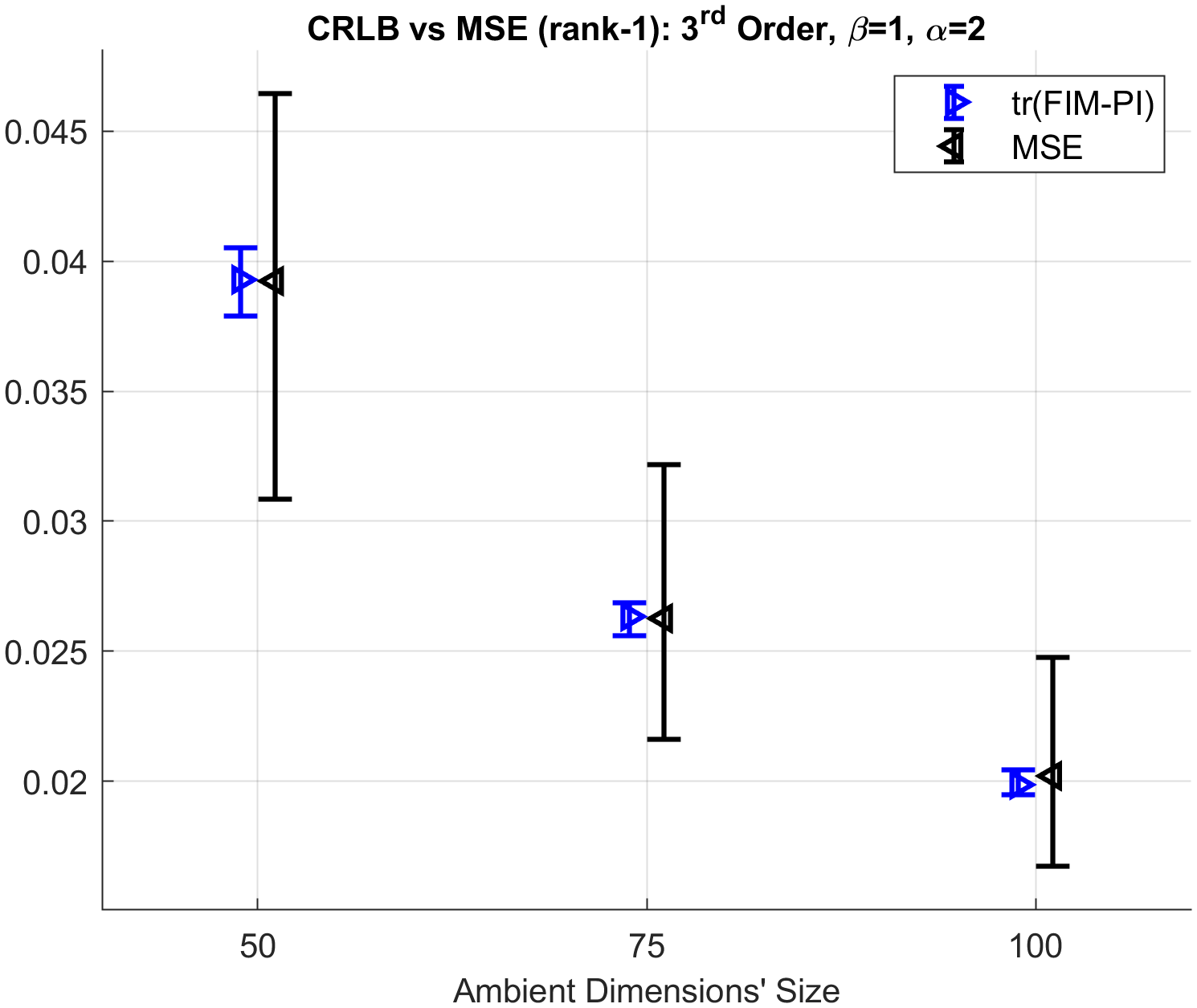} \hspace{50pt}
    \includegraphics[width=0.4\linewidth]{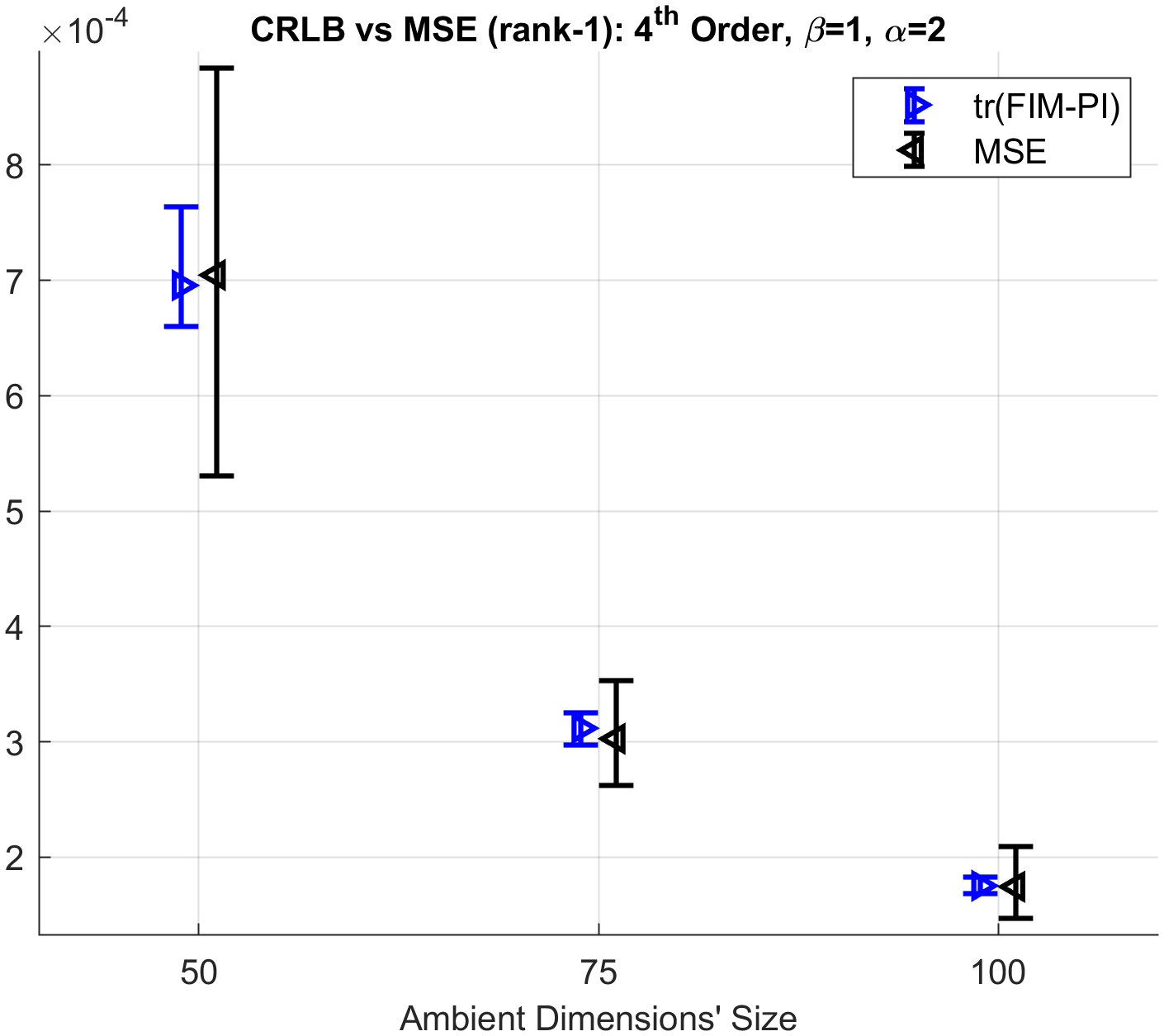}
    \caption{Rank-one experiments: dimension sizes ($I$) vs mean value of tr$\left(\mathcal{I}_{\btheta}^{\dagger}\right)$ and $\|\btheta-\widetilde{\btheta}\|_2^2$, with error bars. Left plot shows the $N=3$ dimensional case and right plot shows $N=4$.}
    \label{plot1}
\end{figure}

\subsection{Larger Rank Cases}

The rank-2 and rank-5 cases follow the same structure, with the key modification that $\boldsymbol{\mathscr{M}}$ is generated as a sum of $R = 2$ or $R = 5$ rank-one components:
\[
\boldsymbol{\mathscr{M}} = \sum_{r=1}^R \mathbf{u}^{(1)}_r \circ \cdots \circ \mathbf{u}^{(N)}_r
\]
where each factor vector $\mathbf{u}^{(n)}_r$ has entries uniformly distributed in
\[
\bigl[\beta^{1/N} / R^{1/N}, \ \alpha^{1/N} / R^{1/N}\bigr],
\]
ensuring that the tensor entries remain within the prescribed range $[\beta,\alpha]$.

As in the rank-one setting, Poisson observations are generated using \texttt{poissrnd}, and the shifted-Poisson regression problem is solved via \texttt{gcp\_opt} with a rank-$R$ constraint and only imposing nonnegative factors. After estimation, the factor vectors of both the ground truth and the recovered tensor are normalized so that all vectors in a given rank-one component have equal $\ell_2$-norm, preserving the total component scale. The factors are stacked into a tall vector and the MSE is computed as before. For the Fisher information term, the factors are again normalized to have equal $\ell_1$-norm per component for consistency to the rank-one case. The FIM is computed using Theorem 3 in \cite{PCPFIMpaper}, pseudo-inverted, and its trace recorded. 
 
 Figures~\ref{plot2} and \ref{plot5} report the empirical MSE and the quantity tr$\left(\mathcal{I}_{\boldsymbol{\theta}}^{\dagger}\right)$, averaged over 50 independent trials, as functions of the ambient dimension $I \in \{50, 75, 100, 125, 150\}$ for both $N = 3$ and $N = 4$. For each dimension, the mean, minimum, and maximum values across realizations are shown. Unlike the rank-one setting, these figures reveal a noticeable gap between the average MSE and the FIM trace term. This indicates that, even in low-rank tensor settings, the rank-constrained estimators do not generally achieve the CRLB. Nevertheless, the spread of the curves suggests that a substantial fraction of the estimators lie near the CRLB, particularly in the large scale $N = 4$ experiments.

To better understand this behavior and highlight the concentration of estimators near the CRLB, we include quantile-based visualizations. Figure~\ref{plotQ} displays the median together with quantile bands: the rank-2 case uses a 0–90\% band, while the rank-5 case uses a more restrictive 0–65\% band due to the increased variability. These plots demonstrate that, although a formal proof of near-efficiency is not yet available, the numerical evidence strongly suggests that most realizations achieve performance close to the CRLB. Specifically, in the rank-2 experiments, only a small number of outliers fall outside the efficient range (beyond the 90th percentile), while in the rank-5 case a greater number of outliers are observed, yet approximately 32 out of 50 trials still lie within the 65\% efficiency band. This indicates an interesting trade-off between tensor rank and the empirical attainability of CRLB-level efficiency.

\begin{figure}
    \centering
    \includegraphics[width=0.4\linewidth]{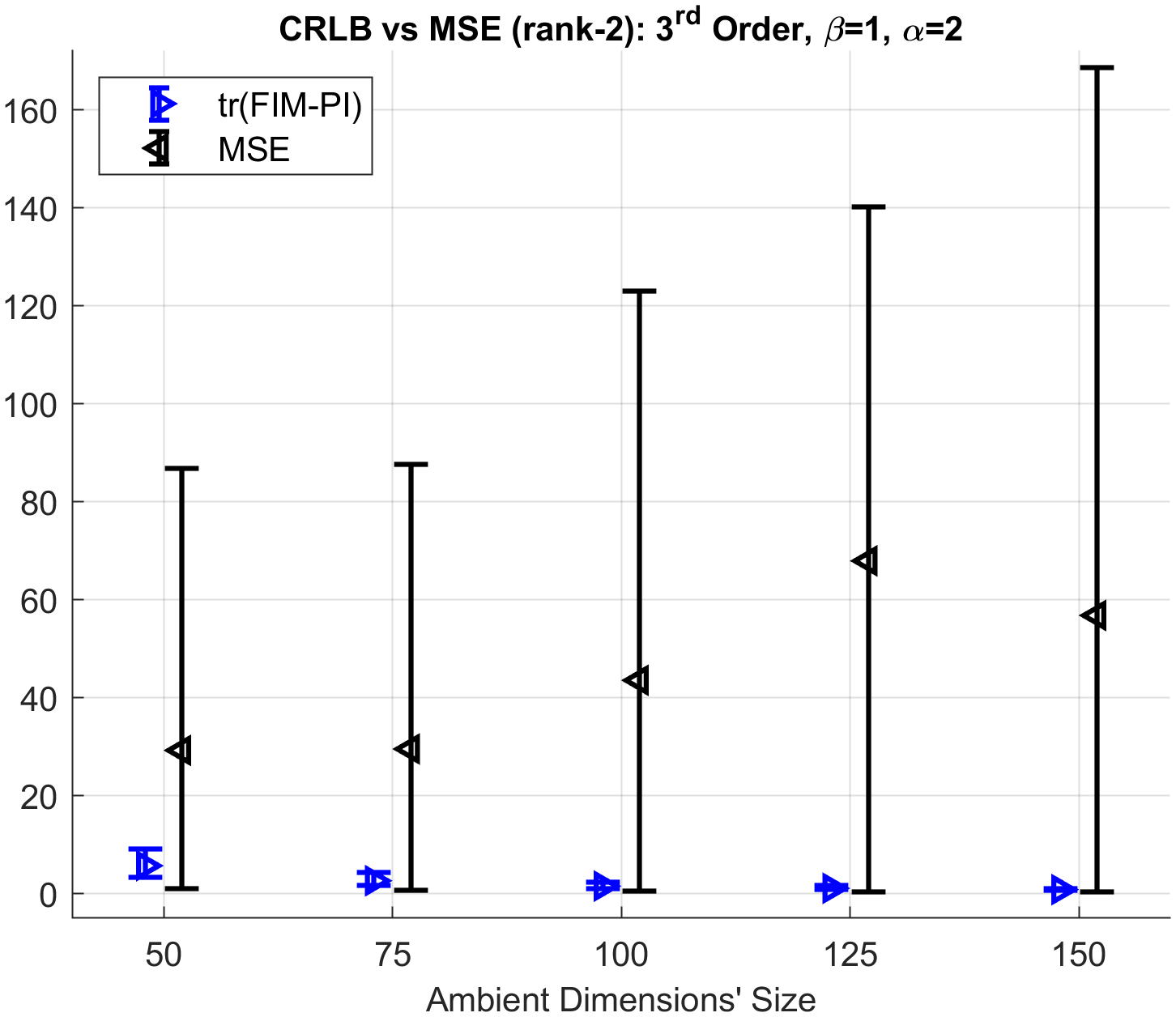} \hspace{50pt}
    \includegraphics[width=0.4\linewidth]{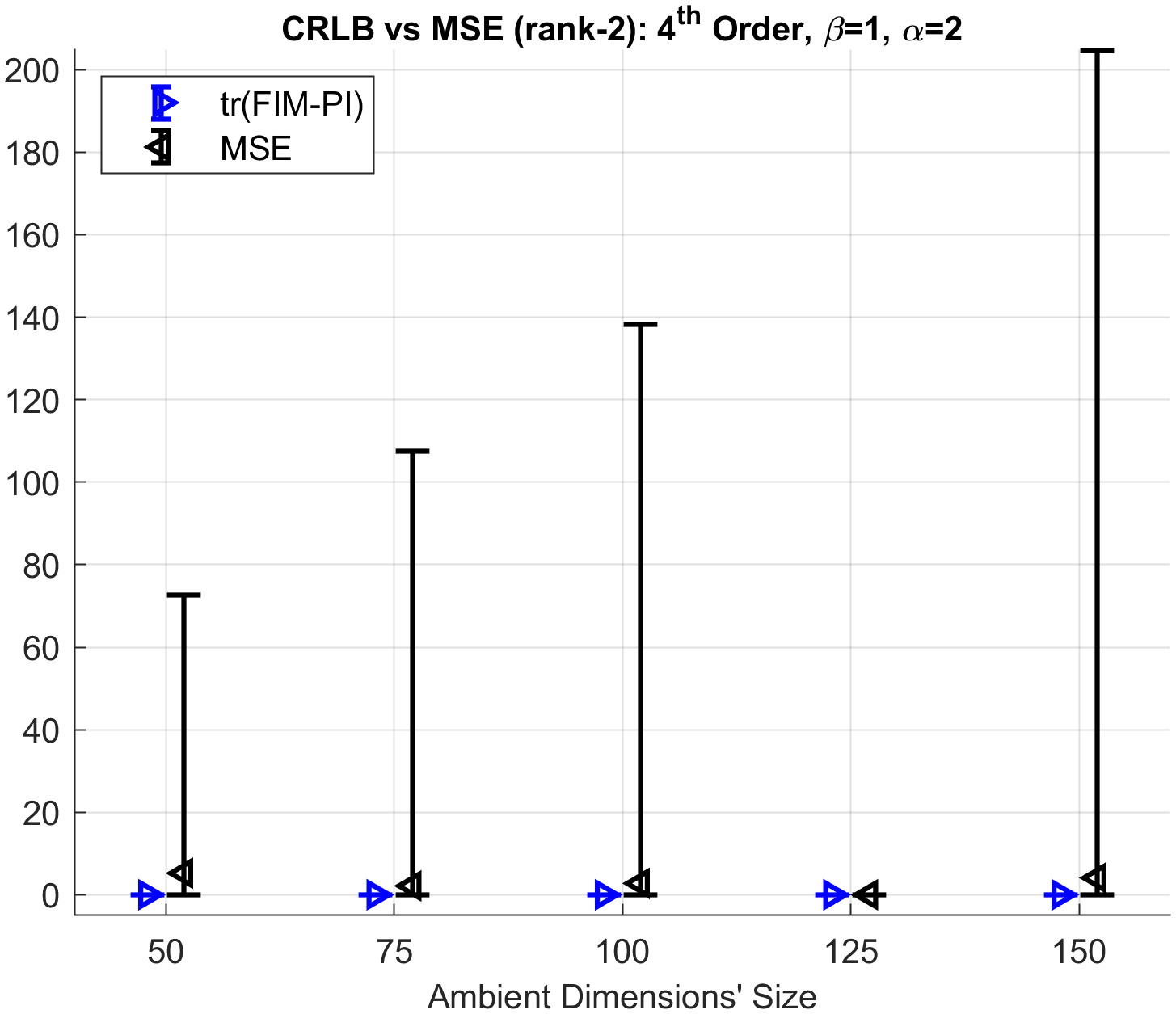}
    \caption{Rank-2 experiments: dimension sizes ($I$) vs mean value of tr$\left(\mathcal{I}_{\btheta}^{\dagger}\right)$ and $\|\btheta-\widetilde{\btheta}\|_2^2$, with error bars. Left plot shows the $N=3$ dimensional case and right plot shows $N=4$.}
    \label{plot2}
\end{figure}

\begin{figure}
    \centering
    \includegraphics[width=0.4\linewidth]{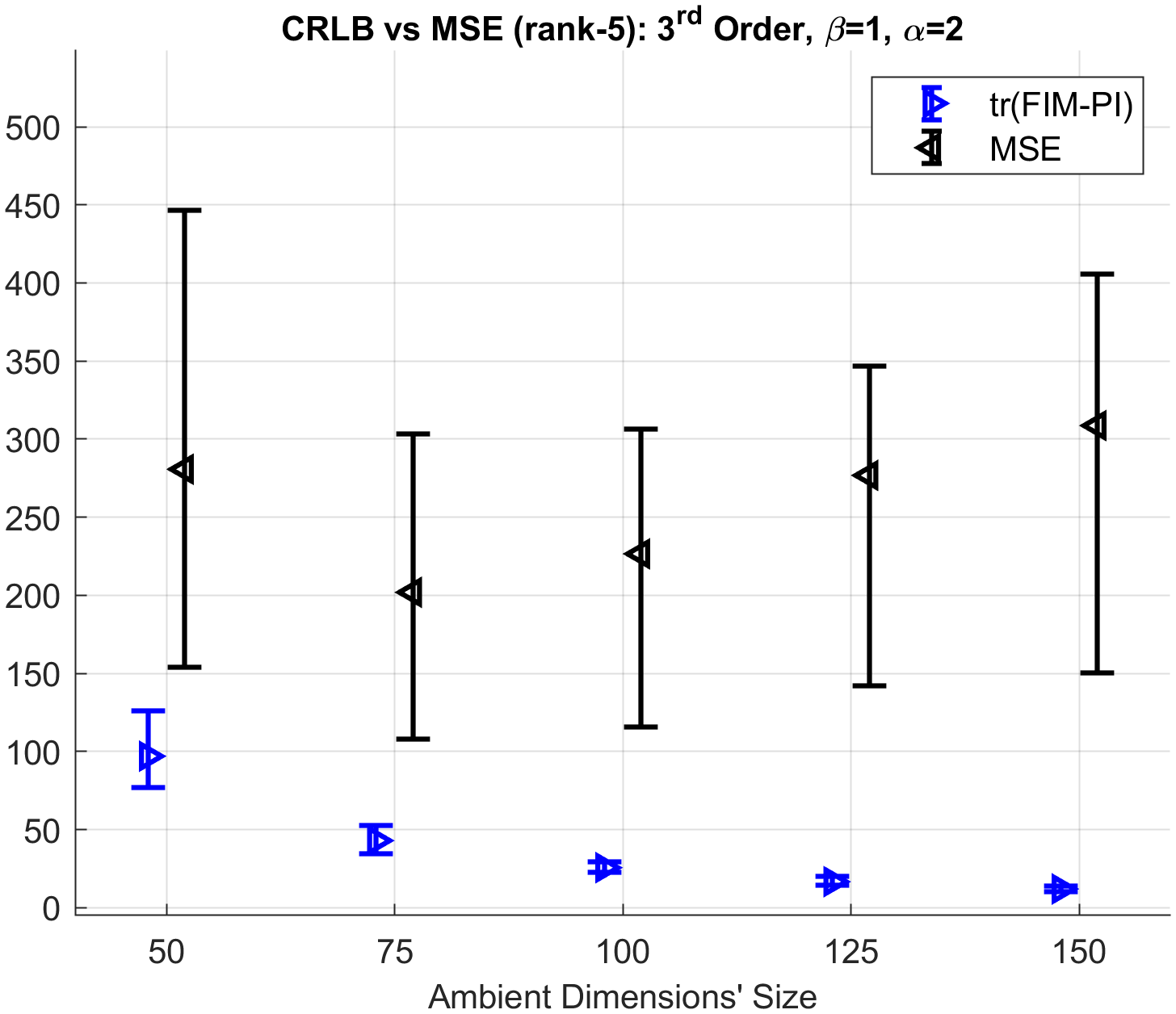} \hspace{50pt}
    \includegraphics[width=0.4\linewidth]{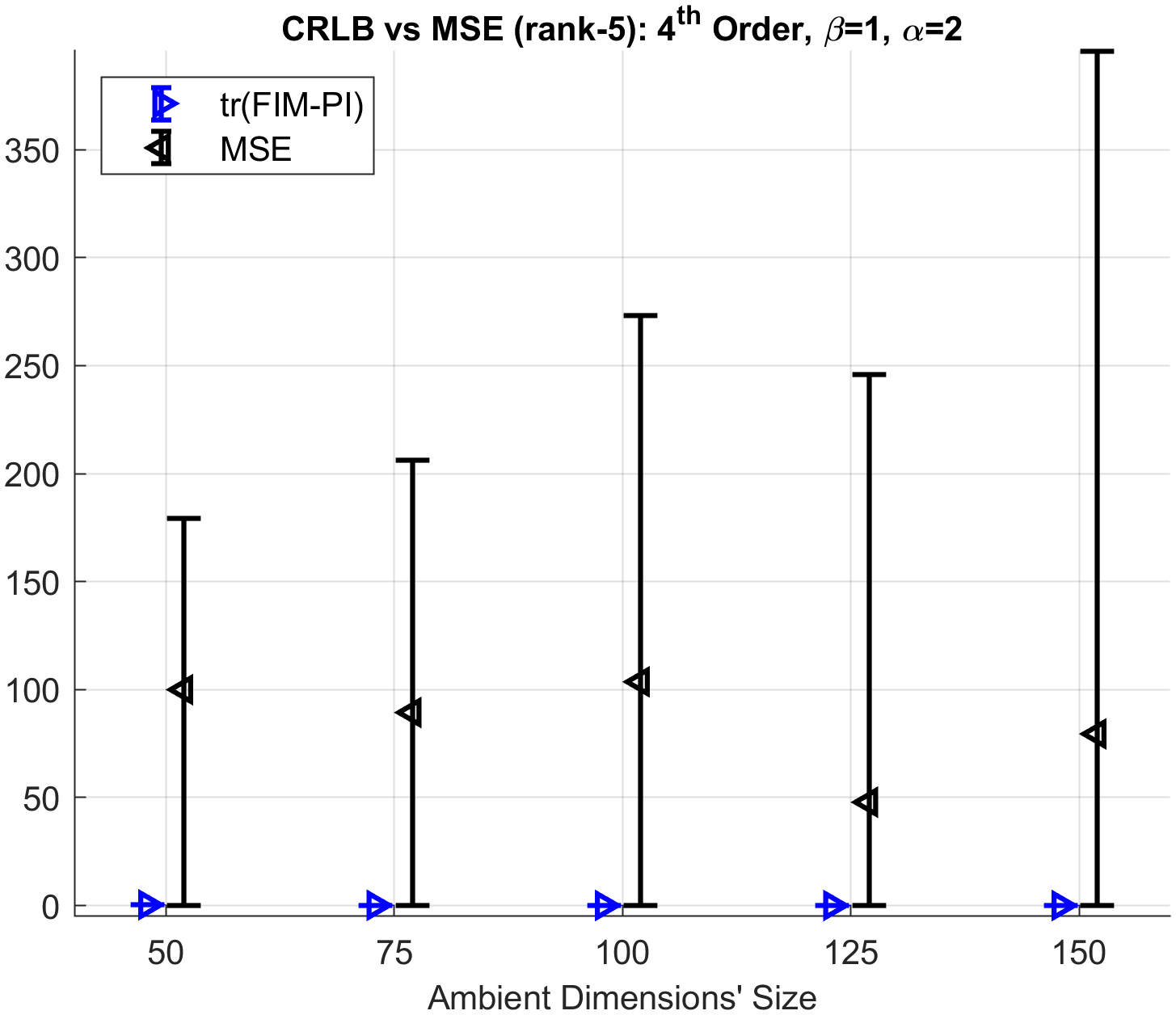}
    \caption{Rank-5 experiments: dimension sizes ($I$) vs mean value of tr$\left(\mathcal{I}_{\btheta}^{\dagger}\right)$ and $\|\btheta-\widetilde{\btheta}\|_2^2$, with error bars. Left plot shows the $N=3$ dimensional case and right plot shows $N=4$.}
    \label{plot5}
\end{figure}

\begin{figure}
    \centering
    \includegraphics[width=0.4\linewidth]{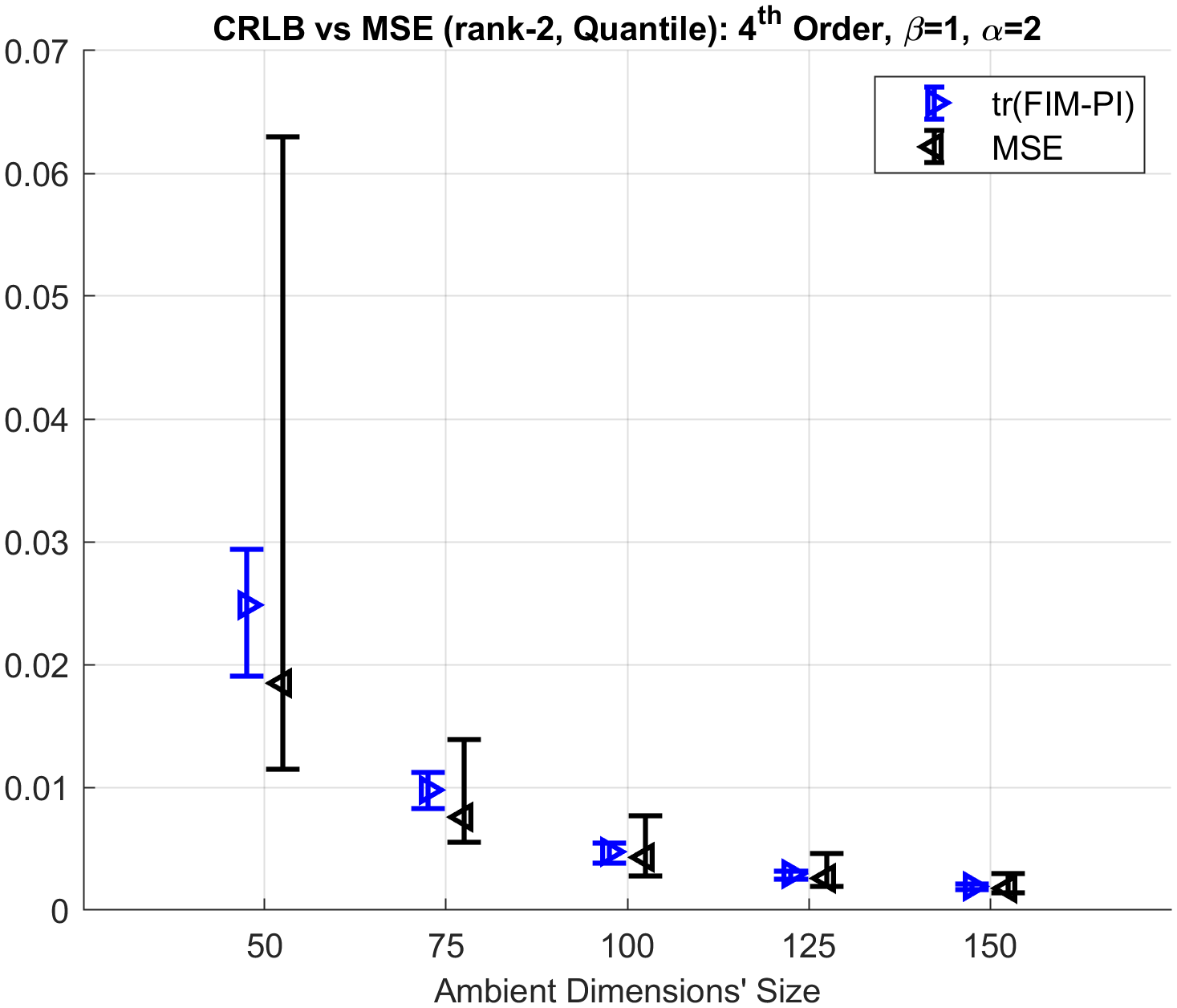} \hspace{50pt}
    \includegraphics[width=0.4\linewidth]{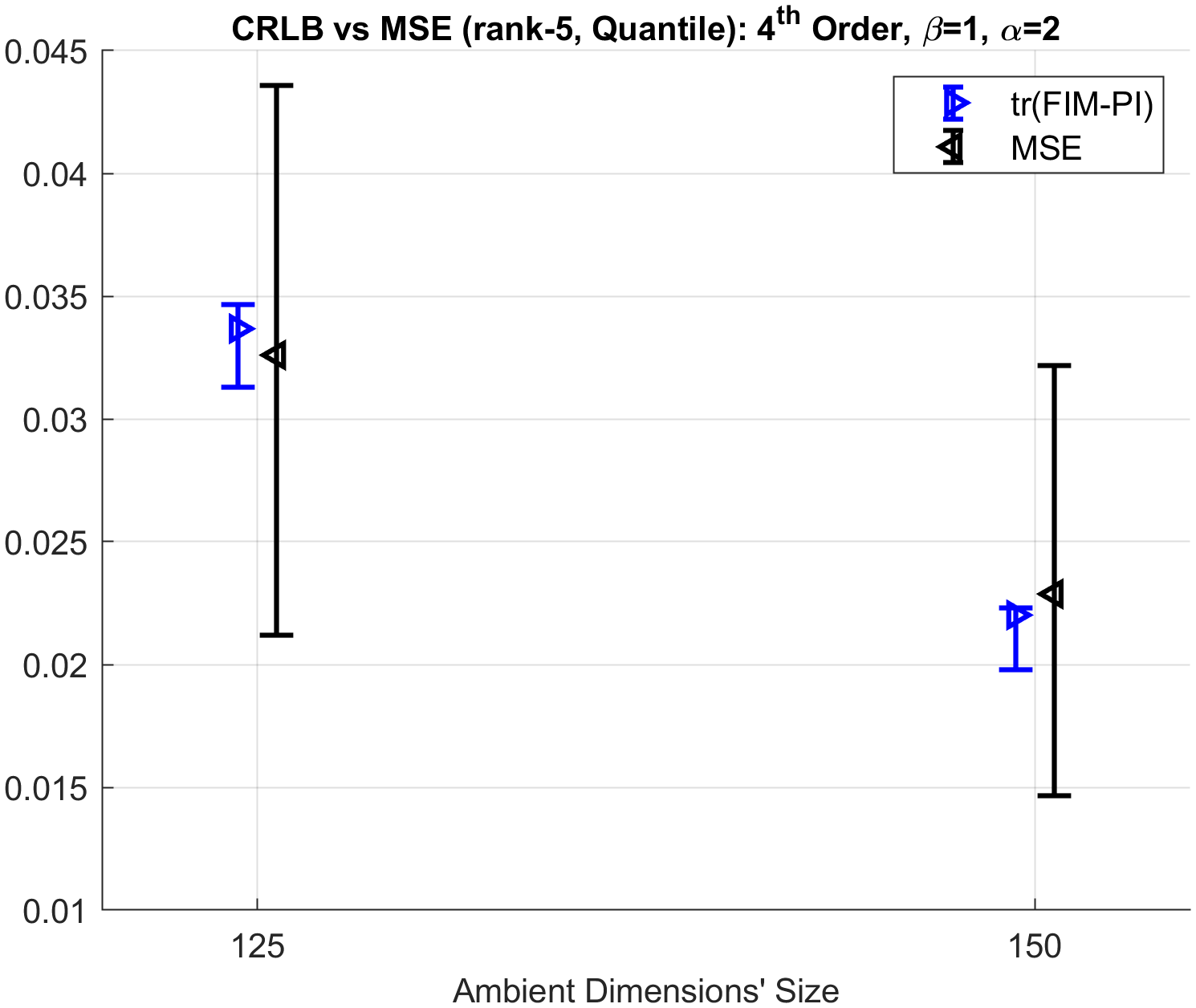}
    \caption{$4^{\mbox{th}}$ order quantile plots: left plot shows the rank-2 case with median and a 0–90\% quantile band and the right plot shows the rank-5 case with median and a 0–65\% quantile band.}
    \label{plotQ}
\end{figure}

\section{Fisher Information Bounds}
\label{FIMsec}

In this section we focus on FIM computations, bounding the pseudo inverse's trace. Our main result in this section is the following:

\begin{theorem}
\label{trFIM}
    Assume that $\boldsymbol{\mathscr{M}}\in\mathcal{S}_1(\beta,\alpha)$ with $\beta>0$, and let $\btheta$ denote its rank-one factors concatenated as a single vector (as in Section \ref{Manal}). Under the Poisson noise model
    \[
   \frac{\beta(I-1)N}{\alpha^{N-1}I^{N-1}} \leq \emph{tr}\left(\mathcal{I}_{\boldsymbol{\theta}}^{\dagger}\right) \leq \frac{\alpha\left(1+(I-1)N\right)}{\beta^{N-1}I^{N-1}}.
    \]
\end{theorem}
Ignoring terms that do not depend on $I$, this result establishes the showcased $\tr(\mathcal{I}_{\boldsymbol{\theta}}^{\dagger}) = \mathcal{O}(I^{2-N})$ in Theorem \ref{simpthm}. 

We now dedicate the remainder of the section for the proof of Theorem \ref{trFIM}. The proof will consist of bounding $\mathcal{F}_{\beta}^{\dagger}\preceq \mathcal{I}_{\boldsymbol{\theta}}^{\dagger} \preceq \mathcal{F}_{\alpha}^{\dagger}$, where $\mathcal{F}_{\beta}$ and $\mathcal{F}_{\alpha}$ are PSD matrices that are easier to analyze. The result will then follow by computing the trace of these matrices.

\begin{proof}[Proof of Theorem \ref{trFIM}]

Without loss of generality, we scale the factors so that
\[
\lambda \coloneqq \|\textbf{u}^{(1)}\|_1 = \|\textbf{u}^{(2)}\|_1 = \cdots = \|\textbf{u}^{(N)}\|_1
\]
which enforces equal factor sums (the factors are non-negative since $\beta>0$). Using Lemma 4 in \cite{PCPFIMpaper}, the FIM of $\boldsymbol{\theta}$ is given in terms of $I\times I$ blocks as
\[
\mathcal{I}_{\boldsymbol{\theta}}  = 
\lambda^{N-1}\begin{bmatrix}
\diag(\textbf{u}^{(1)})^{-1} &\lambda^{-1}\textbf{1}\textbf{1}^* &\dots& \lambda^{-1}\textbf{1}\textbf{1}^* \\
\lambda^{-1}\textbf{1}\textbf{1}^* &\diag(\textbf{u}^{(2)})^{-1} & \dots& \lambda^{-1}\textbf{1}\textbf{1}^* \\
\vdots & \vdots & \ddots & \vdots\\
\lambda^{-1}\textbf{1}\textbf{1}^* & \lambda^{-1}\textbf{1}\textbf{1}^* & \dots & \diag(\textbf{u}^{(N)})^{-1}
\end{bmatrix}\in\mathbb{R}^{NI\times NI},
\]
where $\textbf{1}\in \{1\}^{I}$ is the all-ones vector and diag$(\textbf{x})\in\mathbb{R}^{I\times I}$ is a diagonal matrix with $\textbf{x}\in\mathbb{R}^I$ in the diagonal. Furthermore, Theorem 4 in \cite{PCPFIMpaper} states that $\mathcal{I}_{\boldsymbol{\theta}}$ has rank $(I-1)N +1$ for any $\btheta$ with positive entries.

Let $\textbf{v}\in\mathbb{R}^{NI}$ be any vector. For $n\in \{1,\cdots,N\}$, let $C_n = \{k\in\mathbb{N} \ | \ (n-1)I+1\leq k\leq nI\}$. For fixed $n$, any entry $k\in C_n$ of $\mathcal{I}_{\boldsymbol{\theta}}\textbf{v}$ is given as
\[
\frac{\left(\mathcal{I}_{\boldsymbol{\theta}}\textbf{v}\right)_k}{\lambda^{N-1}} = \frac{v_k}{u_{\bar{k}}^{(n)}} + \sum_{q\in [N]/\{n\}}\lambda^{-1}\sum_{\ell\in C_q}v_{\ell},
\]
where $\bar{k} = k - (n-1)I$. Therefore
\begin{align*}
    &\frac{\textbf{v}^*\mathcal{I}_{\boldsymbol{\theta}}\textbf{v}}{\lambda^{N-1}} = \sum_{n=1}^{N}\sum_{k\in C_n}\frac{v_k\left(\mathcal{I}_{\boldsymbol{\theta}}\textbf{v}\right)_k}{\lambda^{N-1}} = \sum_{n=1}^{N}\sum_{k\in C_n}\left( 
\frac{v_k^2}{u_{\bar{k}}^{(n)}} + v_k\sum_{q\in [N]/\{n\}}\lambda^{-1}\sum_{\ell\in C_q}v_{\ell} \right) \\
&\leq \sum_{n=1}^{N}\sum_{k\in C_n}\left( 
\frac{v_k^2}{\beta} + v_k\sum_{q\in [N]/\{n\}}\lambda^{-1}\sum_{\ell\in C_q}v_{\ell} \right) \coloneqq \frac{\textbf{v}^*\mathcal{F}_{\beta}\textbf{v}}{\lambda^{N-1}},
\end{align*}
where we define
\[
\mathcal{F}_{\beta}  = 
\lambda^{N-1}\begin{bmatrix}
\beta^{-1}\textbf{I} &\lambda^{-1}\textbf{1}\textbf{1}^* &\dots& \lambda^{-1}\textbf{1}\textbf{1}^* \\
\lambda^{-1}\textbf{1}\textbf{1}^* &\beta^{-1}\textbf{I} & \dots& \lambda^{-1}\textbf{1}\textbf{1}^* \\
\vdots & \vdots & \ddots & \vdots\\
\lambda^{-1}\textbf{1}\textbf{1}^* & \lambda^{-1}\textbf{1}\textbf{1}^* & \dots & \beta^{-1}\textbf{I}
\end{bmatrix}
\]
and $\textbf{I}$ denotes the $I\times I$ identity matrix. We note that $\mathcal{F}_{\beta}$ is also of rank $(I-1)N+1$ according to Theorem 4 in \cite{PCPFIMpaper}.

So far, we have shown that
\[
\mathcal{I}_{\boldsymbol{\theta}} \preceq \mathcal{F}_{\beta}.
\]
Since these are of equal rank, Theorem 1 in \cite{pseudo} gives that
\[
\mathcal{F}_{\beta}^{\dagger} \preceq \mathcal{I}_{\boldsymbol{\theta}}^{\dagger}
\]
and therefore tr$(\mathcal{F}_{\beta}^{\dagger})\leq$ tr$(\mathcal{I}_{\boldsymbol{\theta}}^{\dagger})$ can be used as a lower bound. The matrix $\mathcal{F}_{\beta}$ is much simpler to study. Its eigenvectors are given as the columns of the following matrix (defined in terms of $I\times 1$ blocks):
\[
\Gamma = \begin{bmatrix}
\textbf{1} &[\textbf{e}_1 - \textbf{e}_2] &[\textbf{e}_1 + \textbf{e}_2 - 2\textbf{e}_3] &\dots& \big[\sum_{k=1}^{I-1}\textbf{e}_k - (I-1)\textbf{e}_{I}\big] & \textbf{0}&\dots \\
\textbf{1} &\textbf{0} &\textbf{0}& \dots& \textbf{0} & [\textbf{e}_1 - \textbf{e}_2]&\dots \\
\vdots & \vdots & \vdots & \vdots& \vdots& \vdots&\dots\\
\textbf{1} & \textbf{0} &\textbf{0}& \dots & \textbf{0} & \textbf{0}&\dots
\end{bmatrix}\in\mathbb{R}^{NI\times (I-1)N+1},
\]
where $\{\textbf{e}_j\}_{j=1}^{I}\subset\mathbb{R}^I$ denote the canonical basis vectors and the pattern above repeats $N$ times. The respective eigenvalues are
\[
\Lambda_{1} = \lambda^{N-1}\left(\frac{1}{\beta} + \frac{(N-1)I}{\lambda}\right) \ \ \mbox{and} \ \ \Lambda_{r} = \frac{\lambda^{N-1}}{\beta} \ \mbox{for} \ r\geq 2.
\]
Using the factor entry bound $\alpha$ gives
\[
\Lambda_{r} \leq \frac{\alpha^{N-1}I^{N-1}}{\beta}  \ \mbox{for} \ r\geq 2
\]
and we obtain
\[
\tr\left(\mathcal{I}_{\boldsymbol{\theta}}^{\dagger}\right) \geq \tr\left(\mathcal{F}_{\beta}^{\dagger}\right) = \sum_{r=1}^{(I-1)N+1}\frac{1}{\Lambda_r} > \sum_{r=2}^{(I-1)N+1}\frac{1}{\Lambda_r} = \frac{\beta(I-1)N}{\alpha^{N-1}I^{N-1}}.
\]

The upperbound is very similar. The matrix $\mathcal{F}_{\alpha}$, analogous to $\mathcal{F}_{\beta}$ (replacing $
\beta$'s with $\alpha$'s), can be shown to satisfy tr$(\mathcal{F}_{\alpha}^{\dagger})\geq$ tr$(\mathcal{I}_{\boldsymbol{\theta}}^{\dagger})$. This matrix has the same eigenvectors, and similar eigenvalues that replace $
\beta$'s with $\alpha$'s. We can show that
\[
\tr\left(\mathcal{I}_{\boldsymbol{\theta}}^{\dagger}\right) \leq \tr\left(\mathcal{F}_{\alpha}^{\dagger}\right) = 
\lambda^{1-N}\left(\frac{1}{\alpha} + \frac{(N-1)I}{\lambda}\right)^{-1} + \frac{\alpha(I-1)N}{\lambda^{N-1}}
\leq \frac{\alpha\left(1+(I-1)N\right)}{\beta^{N-1}I^{N-1}},
\]
which finishes the proof.

\end{proof}

\section{Mean Squared Error Bounds}
\label{MSEsec}

We now analyze the accuracy of the estimator $\widetilde{\boldsymbol{\mathscr{M}}}$ for parametric inference. This section will prove Theorem \ref{MSEthm}, and then modify the proof in the rank-one case to obtain the following corollary:

\begin{corollary}
\label{ZTPboundfac}
Assume $\boldsymbol{\mathscr{M}}\in\mathcal{S}_1(\beta,\alpha)$ and let $\widetilde{\boldsymbol{\mathscr{M}}}_1$ be given as in \eqref{maxlike} with rank-one factors $\{\widetilde{\textbf{u}}^{(n)}\}_{n=1}^{N}$. Scale the factors such that $\|\textbf{u}^{(1)}\|_2 = \cdots =\|\textbf{u}^{(N)}\|_2$ and $\|\widetilde{\textbf{u}}^{(1)}\|_2 = \cdots =\|\widetilde{\textbf{u}}^{(N)}\|_2$, and let $\btheta,\widetilde{\btheta}$ denote the factors concatenated as a vectors (resp.). There exists an absolute constant $C_3>0$ such that if
    \begin{equation}
    \label{Isize}
        IN^3\log^2(I) \geq \frac{(\beta^N+1)^2\beta^{2(N-1)}(\alpha-\beta)^2}{C_3(\alpha^N+1)^4} \ \ \ \mbox{and} \ \ \ C_1I\log(N)\geq \log(I),
    \end{equation}
where $C_1$ is as in Theorem \ref{MSEthm}, then
    \[
    \mathbb{E}\big\|\boldsymbol{\theta} - \widetilde{\boldsymbol{\theta}}\big\|_2^2 \leq \frac{C_3(\alpha^N+1)^4(N+\delta)^2N^2\log^2(I)}{(\beta^N+1)^2\beta^{2(N-1)}I^{N-2}}.
    \]
\end{corollary}
This corollary is used to obtain the simplified statement $\mathbb{E}\big\|\widetilde{\btheta}-\btheta\big\|_2^2 = \mathcal{O}\left(\frac{\log^2(I)}{I^{N-2}}\right)$ in Theorem \ref{simpthm}, which treats terms that do not depend on $I$ as a constant. We first provide the proof of Theorem \ref{MSEthm}, and then specify the changes needed to obtain Corollary \ref{ZTPboundfac}.

\begin{proof}[Proof of Theorem \ref{MSEthm}]
We will follow the proof of Theorem 1 in \cite{minimaxbin}. By second-order Taylor's theorem, for any $\boldsymbol{\mathscr{T}}\in\mathcal{S}_R(\tilde{\beta},\tilde{\alpha})$ there exists some $\widehat{\boldsymbol{\mathscr{M}}} = \gamma \boldsymbol{\mathscr{M}} + (1-\gamma)\boldsymbol{\mathscr{T}}$ with $\gamma \in [0,1]$ such that
\begin{equation}
\label{expansion}
    \mathcal{L}_{\boldsymbol{\mathscr{X}}}(\boldsymbol{\mathscr{T}}) = \mathcal{L}_{\boldsymbol{\mathscr{X}}}(\boldsymbol{\mathscr{M}}) + \langle \boldsymbol{\mathscr{S}}_{\boldsymbol{\mathscr{X}}}(\boldsymbol{\mathscr{M}}),\boldsymbol{\mathscr{T}}-\boldsymbol{\mathscr{M}}\rangle + \frac{1}{2}\mbox{vec}\left(\boldsymbol{\mathscr{T}}-\boldsymbol{\mathscr{M}}\right)^*H_{\boldsymbol{\mathscr{X}}}(\widehat{\boldsymbol{\mathscr{M}}})\mbox{vec}\left(\boldsymbol{\mathscr{T}}-\boldsymbol{\mathscr{M}}\right).
\end{equation}
Above, $\boldsymbol{\mathscr{S}}_{\boldsymbol{\mathscr{X}}}(\boldsymbol{\mathscr{M}})\in\mathbb{R}^{I\times\cdots\times I}$ is the collection of the score functions evaluated at $\boldsymbol{\mathscr{M}}$ whose entry in the multi-index $\textbf{i}$ is given as
\[
\frac{\partial \mathcal{L}_{\boldsymbol{\mathscr{X}}}(\boldsymbol{\mathscr{M}})}{\partial m_\textbf{i}} = \frac{x_\textbf{i}+1}{m_\textbf{i}+1} - 1.
\]
$H_{\boldsymbol{\mathscr{X}}}(\widehat{\boldsymbol{\mathscr{M}}}) \in\mathbb{R}^{I^N\times I^N}$ is the collection of the Hessian functions evaluated at $\widehat{\boldsymbol{\mathscr{M}}}$. Identifying each $k\in \{1,\cdots,I^N\}$ with a unique multi-index $\textbf{i}_k\in \{1,\cdots,I\}\times\cdots\times\{1,\cdots,I\}$, the entry $(k,\ell)$ of $H_{\boldsymbol{\mathscr{X}}}(\widehat{\boldsymbol{\mathscr{M}}})$ is given as 
\[
\frac{\partial^2 \mathcal{L}_{\boldsymbol{\mathscr{X}}}(\widehat{\boldsymbol{\mathscr{M}}})}{\partial \hat{m}_{\textbf{i}_k}\partial \hat{m}_{\textbf{i}_{\ell}}} = \left\{
     \begin{array}{lr}
       -\frac{x_\textbf{i}+1}{(\hat{m}_\textbf{i}+1)^2}  &\mbox{if} \ k=\ell \ \ (\textbf{i} = \textbf{i}_k = \textbf{i}_{\ell}) \\
       0 &\mbox{if} \ k\neq\ell. \\

     \end{array}
   \right.
\]
We will bound certain terms in \eqref{expansion} and rearrange to obtain our result. Our shifted version of the Poisson loglikelihood gives that
\begin{equation}
    \label{hession}\mbox{vec}\left(\boldsymbol{\mathscr{T}}-\boldsymbol{\mathscr{M}}\right)^*H_{\boldsymbol{\mathscr{X}}}(\widehat{\boldsymbol{\mathscr{M}}})\mbox{vec}\left(\boldsymbol{\mathscr{T}}-\boldsymbol{\mathscr{M}}\right) = -\sum_{\textbf{i}}\frac{x_\textbf{i}+1}{(\hat{m}_\textbf{i}+1)^2}\left(t_{\textbf{i}}-m_{\textbf{i}}\right)^2 \leq \frac{-1}{(\tilde{\alpha}+1)^2}\|\boldsymbol{\mathscr{T}}-\boldsymbol{\mathscr{M}}\|_F^2.
\end{equation}
To bound the linear term in \eqref{expansion}, we use
\[
\langle \boldsymbol{\mathscr{S}}_{\boldsymbol{\mathscr{X}}}(\boldsymbol{\mathscr{M}}),\boldsymbol{\mathscr{T}}-\boldsymbol{\mathscr{M}}\rangle \leq \| \boldsymbol{\mathscr{S}}_{\boldsymbol{\mathscr{X}}}(\boldsymbol{\mathscr{M}})\|\|\boldsymbol{\mathscr{T}}-\boldsymbol{\mathscr{M}}\|_*
\]
where $\|\circ\|$ is the tensor spectral norm and $\|\circ\|_*$ is its dual norm (the tensor nuclear norm), see Lemma 1 in \cite{minimaxbin}. Using Theorem \ref{rankRthm} in the next section we obtain
\begin{equation}
\label{nuclear}
    \|\boldsymbol{\mathscr{T}}-\boldsymbol{\mathscr{M}}\|_* \leq 2R\|\boldsymbol{\mathscr{T}}-\boldsymbol{\mathscr{M}}\|_F,
\end{equation}
since $\boldsymbol{\mathscr{T}}-\boldsymbol{\mathscr{M}}$ is at most rank $2R$. To bound $\| \boldsymbol{\mathscr{S}}_{\boldsymbol{\mathscr{X}}}(\boldsymbol{\mathscr{M}})\|$ we use Lemma 6 in \cite{minimaxbin}. Consider any $\delta>0$, then the entries of $\boldsymbol{\mathscr{S}}_{\boldsymbol{\mathscr{X}}}(\boldsymbol{\mathscr{M}})$ are independent, centered, and satisfy
\[
\max_{\textbf{i}}\bigg|\frac{\partial \mathcal{L}_{\boldsymbol{\mathscr{X}}}(\boldsymbol{\mathscr{M}})}{\partial m_\textbf{i}}\bigg| \leq \max_{\textbf{i}}\frac{(N+\delta)\log(I)+m_\textbf{i}+1}{m_\textbf{i}+1} - 1 \leq \frac{(N+\delta)\log(I)}{\tilde{\beta}+1}.
\]
The first inequality above holds with probability exceeding $1-I^{-\delta}$ and can be shown using Lemma 1 in \cite{PMC}, which gives that $x_\textbf{i} \geq m_\textbf{i} + t$ with probability less that $e^{-t}$ for $t\geq \tilde{\alpha}(e^2-3)$. If $(N+\delta)\log(I)\geq \tilde{\alpha}(e^2-3)$, we may choose $t = (N+\delta)\log(I)$ and obtain that $x_\textbf{i} \geq m_\textbf{i} + t$ for all $\textbf{i}$ with probability less than $I^Ne^{-(N+\delta)\log(I)} = I^{-\delta}$ by a union bound.

 Combining this bound with Lemma 6 in \cite{minimaxbin}, there exists some absolute constants $c_1,c_2>0$ such that with probability at least $1-I^{-\delta}-\exp(-c_1NI\log(N))$
\begin{equation}
    \label{spectral}
    \| \boldsymbol{\mathscr{S}}_{\boldsymbol{\mathscr{X}}}(\boldsymbol{\mathscr{M}})\| \leq c_2\frac{(N+\delta)\log(I)}{\tilde{\beta}+1}\sqrt{NI}.
\end{equation}

Finally, choosing $\boldsymbol{\mathscr{T}} = \widetilde{\boldsymbol{\mathscr{M}}}$ in \eqref{expansion}, since $\mathcal{L}_{\boldsymbol{\mathscr{X}}}(\widetilde{\boldsymbol{\mathscr{M}}}) \geq \mathcal{L}_{\boldsymbol{\mathscr{X}}}(\boldsymbol{\mathscr{M}})$ if we apply \eqref{hession}, \eqref{nuclear}, \eqref{spectral}, and rearrange we have shown that
\begin{equation}
\label{error}
    \big\|\widetilde{\boldsymbol{\mathscr{M}}}-\boldsymbol{\mathscr{M}}\big\|_F \leq \frac{4(\tilde{\alpha}+1)^2c_2(N+\delta)\log(I)R\sqrt{NI}}{\tilde{\beta}+1},
\end{equation}
holds with the prescribed probability. The main statement defines $C_1,C_2$ to absorb $c_1,c_2$ along with other absolute constants.
\end{proof}

We now proceed to the proof of Corollary \ref{ZTPboundfac}. The beginning of the proof is identical to the proof of Theorem \ref{MSEthm}, with a slight modification to obtain a bound analogous to \eqref{error}. Lemma \ref{factorupperbound} from the next section will then translate the derived parameter tensor error \eqref{error} to an error on the CP factors.

\begin{proof}[Proof of Corollary \ref{ZTPboundfac}]

Mimicking the proof of Theorem \ref{MSEthm}, the main difference in the rank-one setting are the bounds imposed on the factors in the search space \eqref{S1}. If $\boldsymbol{\mathscr{T}}\in\mathcal{S}_1(\beta,\alpha)$, notice that $\beta^N\leq t_{\textbf{i}}\leq\alpha^N$. Setting $R=1$, $\tilde{\beta} = \beta^N$, and $\tilde{\alpha} = \alpha^N$ in the definition of $\mathcal{S}_R(\tilde{\beta},\tilde{\alpha})$, we can follow the proof of Theorem \ref{MSEthm} to conclude
\begin{equation}
\label{error1}
    \big\|\widetilde{\boldsymbol{\mathscr{M}}}_1-\boldsymbol{\mathscr{M}}\big\|_F \leq \frac{C_2(\alpha^N+1)^2(N+\delta)\log(I)\sqrt{NI}}{\beta^N+1}
\end{equation}
with probability exceeding $1-I^{-\delta}-\exp(-C_1NI\log(N))$.

By Lemma \ref{factorupperbound}, if we set $\|\textbf{u}^{(1)}\|_2 =  \cdots = \|\textbf{u}^{(N)}\|_2$ and $\|\widetilde{\textbf{u}}^{(1)}\|_2 =  \cdots = \|\widetilde{\textbf{u}}^{(N)}\|_2$ we can conclude
\begin{equation*}
    \big\|\boldsymbol{\theta} - \widetilde{\boldsymbol{\theta}}\big\|_2^2 = \sum_{n=1}^{N}\| \textbf{u}^{(n)} - \tilde{\textbf{u}}^{(n)} \|_2^2 \leq N\frac{C_2^2(\alpha^N+1)^4(N+\delta)^2N\log^2(I)}{(\beta^N+1)^2\beta^{2(N-1)}I^{N-2}} \coloneqq \kappa_{\delta}
\end{equation*}
with probability exceeding $1-I^{-\delta}-\exp(-C_1NI\log(N))$. To bound the expected value, we use Lemma \ref{expbdlemma}. Treat $y = \big\|\boldsymbol{\theta} - \widetilde{\boldsymbol{\theta}}\big\|_2^2$ as a random variable, and notice that $y\in [0,NI(\alpha-\beta)^2]$. We have shown
\[
\mathbb{P}\left(y > \kappa_{\delta}\right) \leq I^{-\delta}+\exp(-C_1NI\log(N))
\]
so that by Lemma \ref{expbdlemma} we will obtain $\EX y \leq 2\kappa_{\delta}$ if
\begin{equation}
\label{expineq}
    I^{-\delta}+\exp(-C_1NI\log(N)) \leq \frac{\kappa_{\delta}}{NI(\alpha-\beta)^2}.
\end{equation}
Under our assumptions, $C_1NI\log(N)\geq \delta\log(I)$ holds with $\delta = N$ and we have that $\exp(-C_1NI\log(N)) \leq I^{-\delta}$. Therefore \eqref{expineq} holds with $\delta = N$ if
\[
I \geq \frac{2(\beta^N+1)^2\beta^{2(N-1)}(\alpha-\beta)^2}{C_2^2(\alpha^N+1)^4(2N)^2N\log^2(I)}.
\]
The final statement defines $C_3 = 2C_2^2$.
\end{proof}

\section{CRLB Proof}
\label{CRLBproof}

In this section we state and prove a general version of Theorem \ref{CRLBsimp}, which allows for biased estimators. For the following result, we refer the reader to the setting and definitions of Section \ref{CRLBsec}.

\begin{theorem}
\label{CRLB}
Consider an estimator $\hat\btheta(\bx)$ of $\btheta$ (not necessarily unbiased). Otherwise, under the conditions of Theorem \ref{CRLBsimp},
$$
\left(\dfrac{\partial}{\partial \btheta} \mathbb{E}\hat\btheta\right)\mathcal{I}_{\btheta}^{\dagger}\left(\dfrac{\partial}{\partial \btheta} \mathbb{E}\hat\btheta\right)^{\top}\preceq \emph{Var}\left(\hat\btheta\right),
$$
where 
$$
\emph{Var}\left(\hat\btheta\right) \coloneqq \EX\Big[\left(\hat\btheta - \EX\hat\btheta\right)\left(\hat\btheta - \EX\hat\btheta\right)^{\top}\Big]
$$ 
is the covariance matrix.
\end{theorem}
In contrast to Theorem \ref{CRLBsimp}, this general CRLB includes terms that consider the bias of $\hat\btheta$. In the unbiased case, notice that $\EX\hat\btheta=\btheta$ so that $\dfrac{\partial}{\partial \btheta} \mathbb{E}\hat\btheta$ is the $p\times p$ identity matrix and Theorem \ref{CRLBsimp} follows. The main novelty of Theorem \ref{CRLB} is our the proof essentially shows that the CRLB is a consequence of the generalized Cauchy-Schwarz inequality.
\begin{theorem}[Theorem 1 in \cite{tripathy2}, article C467]
    \label{tripathy}
    Let $\ba\in\mathbb{R}^{p_1}$ and $\bb\in\mathbb{R}^{p_2}$ be random vectors satisfying $\mathbb{E}\|\ba\|_2<\infty$ and $\mathbb{E}\|\bb\|_2<\infty$. Then:
$$
\mathbb{E}\ba\bb^{\top}\left(
\mathbb{E}\bb\bb^{\top}\right)^{\dagger}
\mathbb{E}\bb\ba^{\top}\preceq \mathbb{E}(\ba\ba^{\top}).
$$
\end{theorem}
The original matrix Cauchy-Schwarz inequality was shown in \cite{tripathi99} and Theorem \ref{tripathy} extends the result to include the generalized inverse. We now proceed to the proof of our CRLB. The proof of Theorem \ref{CRLB} will use the matrix Cauchy-Schwarz inequality, followed by assumptions on the PDF.
\begin{proof}[Proof of Theorem \ref{CRLB}]
Apply Theorem \ref{tripathy} with $\bb =\dfrac{\partial}{\partial \btheta}\log f(\bx;\btheta)$ and $\ba =  \hat\btheta -  \mathbb{E}\hat\btheta$. Then:

\begin{enumerate}
    \item 
    $
    \mathbb{E}\ba\ba^{\top} = \mbox{Var}\left(\hat\btheta\right),
    $
    \item 
    $
    \mathbb{E}\bb\bb^{\top} = \mathcal{I}_{\btheta},
    $
    \item 
    $
    \mathbb{E}\ba\bb^{\top} = \EX\left[  \left(\hat\btheta - \EX\hat\btheta\right)\left(\dfrac{\partial}{\partial \btheta}\log f(\bx;\btheta)\right)^{\top}\right] = \dfrac{\partial}{\partial \btheta} \mathbb{E}\hat\btheta
    $,
\end{enumerate}
where the last equality can be shown using the PDF assumptions from Theorem \ref{CRLBsimp}:
\begin{equation*}
        \begin{split}
             0 = & \dfrac{\partial}{\partial \btheta}\EX
[\hat\btheta - \EX\hat\btheta]
\\=&
             \dfrac{\partial}{\partial \btheta} \int_{\mathbb{R}^n} [\hat\btheta(\bchi) - \EX\hat\btheta] f(\bchi;\btheta) d\bchi
             \\=&              \int_{\mathbb{R}^n}\dfrac{\partial}{\partial \btheta}  \left([\hat\btheta(\bchi) - \EX\hat\btheta] f(\bchi;\btheta)\right) d\bchi 
             \\=& 
             \int_{\mathbb{R}^n}  \left(\dfrac{\partial}{\partial \btheta}[\hat\btheta(\bchi) - \EX\hat\btheta]\right) f(\bchi;\btheta) d\bchi +  
             \int_{\mathbb{R}^n}  [\hat\btheta(\bchi) - \EX\hat\btheta]
         \left(\dfrac{\partial}{\partial \btheta} f(\bchi;\btheta)\right)^{\top}
              d\bchi
             \\=& 
             \int_{\mathbb{R}^n}  \left(\dfrac{\partial}{\partial \btheta}[\hat\btheta(\bchi) - \EX\hat\btheta]\right) f(\bchi;\btheta) d\bchi + 
             \int_{\mathbb{R}^n}  [\hat\btheta(\bchi) - \EX\hat\btheta]\left(\dfrac{\partial}{\partial \btheta}\log f(\bchi;\btheta)\right)^{\top} f(\bchi;\btheta)
              d\bchi
             \\=& 
             \EX \left[\dfrac{\partial}{\partial \btheta}[\hat\btheta - \EX\hat\btheta]\right] + 
             \EX\left[  \left(\hat\btheta - \EX\hat\btheta\right)\left(\dfrac{\partial}{\partial \btheta}\log f(\bx;\btheta)\right)^{\top}\right]
             \\=& 
             \EX \left(\dfrac{\partial}{\partial \btheta}\hat\btheta \right) -
             \dfrac{\partial}{\partial \btheta}
             \EX \hat\btheta  + 
             \EX\left[  \left(\hat\btheta - \EX\hat\btheta\right)\left(\dfrac{\partial}{\partial \btheta}\log f(\bx;\btheta)\right)^{\top}\right].
        \end{split}
    \end{equation*}
    The final term above appears in bullet point 3. We note that, under the assumption of estimability, $\dfrac{\partial}{\partial \btheta}\hat\btheta = 0$ since $\hat\btheta$ cannot depend on $\btheta$. This proves the last equality in bullet 3 and finishes the proof.

\end{proof}

\section{Minimax Lower Bound}
\label{minisec}

In this section, we prove Theorem \ref{Mini}. 
A standard technique for bounding the minimax risk is the Generalized Fano's method \cite[Prop.~5.12]{wainwright2019high}:
\begin{theorem}[Generalized Fano's Method Minimax Lower Bound]\label{thm:fano}
    Given a finite subset of the parameter space $\mathcal{F} \subseteq \mathcal{S}_R$ and a semi-metric loss function $L$, assume that $|\mathcal{F}| = M_\delta$ and that there exists $\delta > 0$ such that
    $$\delta \leq \min_{\Mf^{(i)}, \Mf^{(j)} \in \mathcal{F}, i \neq j} L\left(\Mf^{(i)}, \Mf^{(i)} \right).$$
    Let $J$ denote a random subset of $\Ff$ and let $W$ denote the observed data. 
    If $I(W, J)$ is the mutual information between $W$ and a random variable given by a mixture of the distributions in $\Ff$ with indices in $J$, 
    we have that 
    $$\inf_{\widehat{\Mf}} \sup_{\Mf \in \mathcal{S}_R} \EX \Phi\left( L\left(\widehat{\Mf}, \Mf\right) \right) \geq \Phi\left(\frac{\delta}{2}\right) \left(1 - \frac{I(W, J) + \log 2}{\log M_\delta}\right),$$
    for any strictly increasing function $\Phi : [0, \infty) \mapsto [0, \infty)$.
\end{theorem}

We associate each parameter tensor $\Mf^{(i)}\in \mathcal{F}$ with the corresponding probability distribution $P^{(i)} \in \Pp_\mathcal{F}$, in this case, a joint Poisson distribution over the entries of the element. 
We omit the definition of mutual information here, but note that by \cite[Eq.~(13)]{KLbound}
$$I(W, J) \leq \frac{1}{\binom{M_{\delta}}{2} } \sum_{P^{(i)}, P^{(j)} \in \Pp_\mathcal{F},i\neq j} D_{KL}\left(P^{(i)} \| P^{(j)}\right) \leq \max_{P^{(i)}, P^{(j)} \in \Pp_\mathcal{F},i\neq j} D_{KL}\left(P^{(i)} \| P^{(j)}\right),$$
where $D_{KL}$ denotes the Kullback-Leibler (KL) divergence.
For tensors with independent Poisson entries we may write:
\[
D_{KL}\left( P^{(i)} \| P^{(j)} \right) = \sum_{\textbf{k}\in\{1,\cdots,I\}^N} \left[ m_{\textbf{k}}^{(i)} \log \frac{m_{\textbf{k}}^{(i)}}{m_{\textbf{k}}^{(j)}} - m_{\textbf{k}}^{(i)} + m_{\textbf{k}}^{(j)}\right].
\] 
This observation enables us to write the minimax lower bound in terms of the KL divergence:
\begin{corollary}\label{thm:fano_minimax}
    Under the assumptions of Theorem \ref{thm:fano}, if we further have $0 < \gamma < \infty$ such that
    $$\gamma \geq \max_{i \neq j; P^{(i)}, P^{(j)} \in \Pp_\mathcal{F}} D_{KL}\left(P^{(i)} \| P^{(j)}\right),$$
    we may write 
    $$\inf_{\widehat{\Mf}} \sup_{\Mf \in \mathcal{S}_R} \EX \Phi\left( L\left(\widehat{\Mf}, \Mf\right) \right) \geq \Phi\left(\frac{\delta}{2}\right) \left(1 - \frac{ \gamma + \log 2}{\log M_\delta}\right).$$
\end{corollary}
As a point of interest, if we were to observe $n > 1$ realizations (samples) of the tensor, the decoupling property of the KL divergence would scale the KL divergence by $n$ (see \cite[Eq.~15.11(a)]{wainwright2019high}), and we would replace $\gamma$ by $n \gamma$ in Corollary \ref{thm:fano_minimax}.

To obtain Theorem \ref{Mini}, we will choose $\Phi(x) = x^2$ and the Frobenius norm of $\widehat{\Mf} - \Mf$ as a the loss function. 
To establish our minimax bound, we will control $\delta$ so that 
$$\log M_\delta \geq 2 (\gamma + \log 2),$$
leading to 
$$\left(1 - \frac{\gamma + \log 2}{\log M_\delta}\right) \geq \frac{1}{2},$$
and a minimax lower bound of $\frac{1}{2} \Phi(\delta / 2) = \frac{1}{8} \delta^2$.
We have three steps in our analysis:
\begin{enumerate}
 \item Define a packing set $\Ff$ with separation $\delta$.
\item Compute the upper bound $\gamma$ required in Corollary \ref{thm:fano_minimax}.
    \item Select $\delta$ such that $\log M_\delta \geq 2 (\gamma + \log 2)$.
\end{enumerate}

The proof applies Lemmas \ref{lem:minimaxpacking} and \ref{lem:kl_pois}, which we state afterwards. We refer the reader to Section \ref{minimaxlemmas} for the required lemmas and their proofs.

\begin{proof}[Proof of Theorem \ref{Mini}]
We begin by applying Lemma \ref{lem:minimaxpacking} to create a packing set in $\mathcal{S}_R(\tilde{\beta},\tilde{\alpha})$. We then have the following: for some constant $\epsilon \in (0, 1]$ (to be specified later), we have a packing set $\mathcal{F}$ with $M_\delta = |\mathcal{F}| \geq 2^{IR/ 8}$ and spacing $\delta = \frac{\epsilon}{4} (\tilde{\alpha} - \tilde{\beta})\sqrt{I^N}$. 

We next apply Lemma \ref{lem:kl_pois} to upper bound the KL divergence between the elements of the packing set. We conclude that since the spacing of elements in the packing set is bounded above by $\epsilon (\tilde{\alpha} - \tilde{\beta}) \sqrt{I^N}$ (by Lemma \ref{lem:minimaxpacking}), the KL divergence is bounded above by
$$\gamma = \epsilon^2 \frac{(\tilde{\alpha} - \tilde{\beta})^2}{\tilde{\beta}} I^N.$$

We now need to tune $\delta$ so that $\log M_\delta \geq 2(\gamma + \log 2)$, which
we accomplish by adjusting the parameter $\epsilon$. 
That is, we choose $\epsilon \in (0, 1]$ so that 
$$\frac{IR\log 2}{8} \geq 2 \left(\epsilon^2 \frac{(\tilde{\alpha} - \tilde{\beta})^2}{\tilde{\beta}} I^N + \log 2 \right).$$
Equivalently,
$$\epsilon^2 \leq \frac{(\frac{IR}{16} - 1)}{I^N} \frac{\tilde{\beta}\log 2}{(\tilde{\alpha} - \tilde{\beta})^2}.$$
For $\epsilon > 0$, we require $IR / 16 - 1 > 0$, or $IR > 16$---a slightly stronger condition than the $IR > 8$ required by the construction of the packing set in Lemma \ref{lem:minimaxpacking}. 
We also require that $\epsilon^2 \leq 1$, which imposes condition \eqref{condmini} in the statement of the theorem. 

It remains to compute the final minimax bound, $\delta^2 / 8$:
\begin{align*}
    \frac{1}{2} \Phi(\delta / 2) &= \frac{1}{8} \delta^2 \\
    &= \frac{1}{8} \left(\frac{\epsilon^2}{16} (\tilde{\alpha} - \tilde{\beta})^2 I^N\right) \\
    &= \frac{\log 2}{128} \frac{(\frac{IR}{16} - 1)}{I^N} \frac{\tilde{\beta}}{(\tilde{\alpha} - \tilde{\beta})^2} (\tilde{\alpha} - \tilde{\beta})^2 I^N \\
    &= \frac{\tilde{\beta}\log 2}{128} \left(\frac{IR}{16} - 1 \right).
\end{align*}     
\end{proof}

\section{Required Lemmas}
\label{auxlemmas}

In this section we provide some necessary lemmas. Subsection \ref{mselem} states results needed for the MSE bounds in Section \ref{MSEsec}, while Subsection \ref{minimaxlemmas} focuses on lemmas needed for the minimax bound in Section \ref{minisec}. The proofs are provided in Subsection \ref{lemmasproofs}.

\subsection{MSE Lemmas}
\label{mselem}

The first lemma translates a bound on the tensor error to a bound on rank-one factor error.

\begin{lemma}
\label{factorupperbound}
    Let $\boldsymbol{\mathscr{M}}$ and $\widetilde{\boldsymbol{\mathscr{M}}}$ be $N$-way rank-one tensors with
    \[
\boldsymbol{\mathscr{M}}= \textbf{u}^{(1)}\circ\cdots \circ \textbf{u}^{(N)} \ \ \mbox{and} \ \ \widetilde{\boldsymbol{\mathscr{M}}}= \tilde{\textbf{u}}^{(1)}\circ\cdots \circ \tilde{\textbf{u}}^{(N)},
\]
where $\textbf{u}^{(n)}, \tilde{\textbf{u}}^{(n)}\in [\beta,\infty)^{I}$ for all $n\in[N]$. Assume $\|\textbf{u}^{(1)}\|_2 = \|\textbf{u}^{(2)}\|_2 = \cdots = \|\textbf{u}^{(N)}\|_2$ and $\|\tilde{\textbf{u}}^{(1)}\|_2 = \|\tilde{\textbf{u}}^{(2)}\|_2 = \cdots = \|\tilde{\textbf{u}}^{(N)}\|_2$.

If
\[
\Big\| \boldsymbol{\mathscr{M}} - \widetilde{\boldsymbol{\mathscr{M}}} \Big\|_F^{2} \leq \epsilon
\]
for some $\epsilon>0$, then
\[
\| \textbf{u}^{(n)} - \tilde{\textbf{u}}^{(n)} \|_2^2 \leq \frac{\epsilon}{\beta^{2(N-1)} I^{N-1}}
\]
for all $n \in [N]$.
\end{lemma}

We next present a result bounding the nuclear norm of a CP rank $R$ tensor in terms of the Frobenius norm. The following result improves upon previous tensor nuclear norm bounds in the literature \cite{minimaxbin,NNbound}. We present it as a theorem since it may be of interest on its own.

 \begin{theorem}
	\label{rankRthm}
	If $\boldsymbol{\mathscr{X}}\in\mathbb{C}^{I\times \cdots\times I}$ is a CP rank R tensor, then
 \begin{equation}
 \label{TrankR}
     \|\boldsymbol{\mathscr{X}}\|_{*} \leq R\|\boldsymbol{\mathscr{X}}\|_F.
 \end{equation}
\end{theorem}

We end this section with a lemma that bounds the expected value based on a tail distribution bound. This lemma helps translate high probability error bounds to MSE bounds.

\begin{lemma}
    \label{expbdlemma}
    Let $y\in [0,L]$ be a continuous random variable. If for some $\delta$, it holds that $\mathbb{P}(y\geq \delta)\leq \frac{\delta}{L}$ then $\EX y \leq 2\delta$.
\end{lemma}

\subsection{Minimax Lemmas}
\label{minimaxlemmas}

First, we recall the classical Varshamov-Gilbert bound, taken from \cite[Lemma~7]{minimaxbin}.
\begin{lemma}[Varshamov-Gilbert bound]\label{lem:vg}
    Let $\Omega = \left\{(w_1, w_2, \ldots, w_m) : w_i \in \{0, 1\}\right\} \subseteq \mathbb{R}^m$, and take $m > 8$. 
    Then, there exists a subset $\left\{\ww^{(0)}, \ww^{(1)}, \ldots, \ww^{(M)}\right\}$ of $\Omega$ such that $\ww^{(0)}$ is the zero vector and for $0 \leq j < k \leq M$,
    $$\left\| \ww^{(j)} - \ww^{(k)} \right\|_0 \geq \frac{m}{8}.$$
\end{lemma}

We next use the Varshamov-Gilbert bound to construct a packing set $\mathcal{F} \subseteq \mathcal{S}_R(\tilde{\beta}, \tilde{\alpha})$ with $\tilde{\beta} = 0$ (no lower bound on the entries of the tensor).
Our approach is a modification of \cite[Lemma~9]{lee2020tensor} and \cite[Lemma~8]{minimaxbin}.
\begin{lemma}[Minimax Packing Set]\label{lem:minimaxpacking}
    Assume that $R \leq I$ and $IR > 8$. For a constant $\epsilon \in (0, 1]$ and bounds $0 < \tilde{\beta} < \tilde{\alpha}$, there is a finite set $\mathcal{F} \subset \mathcal{S}_R(\tilde{\beta}, \tilde{\alpha})$ such that:
    \begin{enumerate}
        \item The cardinality satisfies $|\mathcal{F}| \geq 2^{IR / 8}$;
        \item For any pair of distinct elements $\boldsymbol{\mathscr{M}}^{(i)}, \boldsymbol{\mathscr{M}}^{(j)} \in \mathcal{F}$, 
        $$\left\| \boldsymbol{\mathscr{M}}^{(i)} - \boldsymbol{\mathscr{M}}^{(j)} \right\|_F \geq \frac{\epsilon}{4} (\tilde{\alpha} - \tilde{\beta}) \sqrt{I^N};$$
        \item For any pair of distinct elements $\boldsymbol{\mathscr{M}}^{(i)}, \boldsymbol{\mathscr{M}}^{(j)} \in \mathcal{F}$, 
        $$\left\| \boldsymbol{\mathscr{M}}^{(i)} - \boldsymbol{\mathscr{M}}^{(j)} \right\|_F \leq  \epsilon (\tilde{\alpha} -\tilde{\beta}) \sqrt{I^N}.$$
    \end{enumerate}
\end{lemma}

We next state a useful lemma about the multivariate Poisson distribution, taken from \cite[Lemma~19]{jiang2015minimax} but rephrased to fit our context and notation.

\begin{lemma} \label{lem:kl_pois}
As in Section \ref{minisec}, associate each Poisson parameter tensor $\Mf^{(i)}$ with a corresponding multivariate Poisson distribution $P^{(i)}$. Then the Kullback-Leibler (KL) divergence is bounded as:
    $$D_{KL}\left( P^{(i)} \| P^{(j)} \right) \leq \frac{1}{\min_{\emph{\textbf{i}}} m_{\emph{\textbf{i}}}^{(j)}}\Big\|\Mf^{(i)} - \Mf^{(j)}\Big\|_F^2 \leq \frac{1}{\tilde{\beta}}\Big\|\Mf^{(i)} - \Mf^{(j)}\Big\|_F^2.$$
\end{lemma}
This lemma will allow us to upper bound the KL divergence between elements of the packing set by an upper bound on the distance between elements. 

\subsection{Proof of Lemmas}
\label{lemmasproofs}

\subsubsection{Proof of Lemma \ref{factorupperbound}}
\begin{proof}[Proof of Lemma \ref{factorupperbound}]

Under our assumptions, notice that
\[
\Big\| \boldsymbol{\mathscr{M}} - \widetilde{\boldsymbol{\mathscr{M}}} \Big\|_F^{2} = \prod_{n=1}^{N}\|\textbf{u}^{(n)}\|_2^2 + \prod_{n=1}^{N}\|\tilde{\textbf{u}}^{(n)}\|_2^2 -2\prod_{n=1}^{N}\langle \textbf{u}^{(n)},\tilde{\textbf{u}}^{(n)}\rangle \leq \epsilon.
\]
Since $\langle \textbf{u}^{(n)},\tilde{\textbf{u}}^{(n)}\rangle \geq 0$ for all $n$, we obtain
\begin{align*}
   &\|\textbf{u}^{(n)}-\tilde{\textbf{u}}^{(n)}\|_2^2 = \|\textbf{u}^{(n)}\|_2^2 + \|\tilde{\textbf{u}}^{(n)}\|_2^2 -2\langle \textbf{u}^{(n)},\tilde{\textbf{u}}^{(n)}\rangle \\ 
   &\leq \|\textbf{u}^{(n)}\|_2^2 + \|\tilde{\textbf{u}}^{(n)}\|_2^2 + \prod_{k\neq n}\frac{\epsilon}{\langle \textbf{u}^{(k)},\tilde{\textbf{u}}^{(k)}\rangle} - \|\textbf{u}^{(n)}\|_{2}^{2}\prod_{k\neq n}\frac{\|\textbf{u}^{(k)}\|_{2}^{2}}{\langle \textbf{u}^{(k)},\tilde{\textbf{u}}^{(k)}\rangle} - \|\tilde{\textbf{u}}^{(n)}\|_{2}^{2}\prod_{k\neq n}\frac{\|\tilde{\textbf{u}}^{(k)}\|_{2}^{2}}{\langle \textbf{u}^{(k)},\tilde{\textbf{u}}^{(k)}\rangle} \\
   &= \prod_{k\neq n}\frac{\epsilon}{\langle \textbf{u}^{(k)},\tilde{\textbf{u}}^{(k)}\rangle} + \|\textbf{u}^{(n)}\|_{2}^{2}\left(1-\prod_{k\neq n}\frac{\|\textbf{u}^{(k)}\|_{2}^{2}}{\langle \textbf{u}^{(k)},\tilde{\textbf{u}}^{(k)}\rangle}\right) + \|\tilde{\textbf{u}}^{(n)}\|_{2}^{2}\left(1-\prod_{k\neq n}\frac{\|\tilde{\textbf{u}}^{(k)}\|_{2}^{2}}{\langle \textbf{u}^{(k)},\tilde{\textbf{u}}^{(k)}\rangle}\right) \\
   &\leq \prod_{k\neq n}\frac{\epsilon}{\langle \textbf{u}^{(k)},\tilde{\textbf{u}}^{(k)}\rangle} + \|\textbf{u}^{(n)}\|_{2}^{2}\left(1-\prod_{k\neq n}\frac{\|\textbf{u}^{(k)}\|_{2}}{\|\tilde{\textbf{u}}^{(k)}\|_2}\right) + \|\tilde{\textbf{u}}^{(n)}\|_{2}^{2}\left(1-\prod_{k\neq n}\frac{\|\tilde{\textbf{u}}^{(k)}\|_{2}}{\| \textbf{u}^{(k)}\|_2}\right),
\end{align*}
where the last line applies the Cauchy–Schwarz inequality. 

Assume WLOG that $\|\widetilde{\boldsymbol{\mathscr{M}}} \|_F \leq \| \boldsymbol{\mathscr{M}}\|_F$, and note this implies that $\|\tilde{\textbf{u}}^{(k)}\|_{2} \leq \|\textbf{u}^{(k)}\|_{2}$ for all $k$.
Then the last term in the last line above is non-negative. Using $\|\tilde{\textbf{u}}^{(n)}\|_{2} \leq \|\textbf{u}^{(n)}\|_{2}$ we obtain
\[
     \|\textbf{u}^{(n)}-\tilde{\textbf{u}}^{(n)}\|_2^2 \leq \prod_{k\neq n}\frac{\epsilon}{\langle \textbf{u}^{(k)},\tilde{\textbf{u}}^{(k)}\rangle} + \|\textbf{u}^{(n)}\|_{2}^{2}\left(2-\prod_{k\neq n}\frac{\|\textbf{u}^{(k)}\|_{2}}{\|\tilde{\textbf{u}}^{(k)}\|_2} -\prod_{k\neq n}\frac{\|\tilde{\textbf{u}}^{(k)}\|_{2}}{\| \textbf{u}^{(k)}\|_2}\right) \leq \prod_{k\neq n}\frac{\epsilon}{\langle \textbf{u}^{(k)},\tilde{\textbf{u}}^{(k)}\rangle}
\]
where the last inequality follows from the fact that the function $f(x,y) = \frac{x}{y}+\frac{y}{x}$ is greater than or equal to 2 for all $x,y> 0$. If $\| \boldsymbol{\mathscr{M}}\|_F \leq \|\widetilde{\boldsymbol{\mathscr{M}}}\|_F$, an analogous argument can be applied to obtain the same conclusion.

To finish, since $\textbf{u}^{(k)},\tilde{\textbf{u}}^{(k)}\in [\beta,\infty)^{I}$ we have $\langle \textbf{u}^{(k)},\tilde{\textbf{u}}^{(k)}\rangle \geq \beta^2 I$ and therefore
\[
\|\textbf{u}^{(n)}-\tilde{\textbf{u}}^{(n)}\|_2^2 \leq \frac{\epsilon}{\beta^{2(N-1)} I^{N-1}}.
\]
    
\end{proof}

\subsubsection{Proof of Theorem \ref{rankRthm}}
\begin{proof}[Proof of Theorem \ref{rankRthm}]

Consider a rank $R$ tensor $\boldsymbol{\mathscr{X}}\in\mathbb{C}^{I_1\times\cdots\times I_N}$ and write its CPD as
\begin{equation}
\label{CPX}
\boldsymbol{\mathscr{X}} = \sum_{r=1}^{R}a_r\boldsymbol{\mathscr{X}}_r
\end{equation}
where $\{a_r\}_{r=1}^{R}\subset \mathbb{C}$ and each $\boldsymbol{\mathscr{X}}_r$ is a rank-one unit norm tensor. It is important to note that we may assume the $\{\boldsymbol{\mathscr{X}}_r\}_{r=1}^{R}$ are linearly independent, otherwise we can obtain a CPD of lower rank. By the triangle inequality
\[
\|\boldsymbol{\mathscr{X}}\|_{*} \leq \sum_{r=1}^{R}|a_r|\|\boldsymbol{\mathscr{X}}_r\|_{*} = \sum_{r=1}^{R}|a_r|.
\]
The main component of the proof will show that $|a_r| \leq \|\boldsymbol{\mathscr{X}}\|_F$ for each $r\in \{1,\cdots,R\}$ and the result will follow.

Considering the CPD of $\boldsymbol{\mathscr{X}}$ above, WLOG assume that
\[
|a_R| = \max_{r}|a_r|.
\]
We now follow the Gram-Schmidt process to produce an orthogonal basis that spans the linear subspace spanned by $\{\boldsymbol{\mathscr{X}}_r\}_{r=1}^{R}$. The new orthogonal basis is defined as
\begin{align*}
    &\widetilde{\boldsymbol{\mathscr{X}}}_1 = \boldsymbol{\mathscr{X}}_1, \\
    &\widetilde{\boldsymbol{\mathscr{X}}}_2 = \boldsymbol{\mathscr{X}}_2 - \frac{\langle\boldsymbol{\mathscr{X}}_2, \widetilde{\boldsymbol{\mathscr{X}}}_1\rangle}{\big\|\widetilde{\boldsymbol{\mathscr{X}}}_{1}\big\|_F} \widetilde{\boldsymbol{\mathscr{X}}}_1, \\
    &\vdots \\
    &\widetilde{\boldsymbol{\mathscr{X}}}_R = \boldsymbol{\mathscr{X}}_R - \sum_{k=1}^{R-1}\frac{\langle\boldsymbol{\mathscr{X}}_R, \widetilde{\boldsymbol{\mathscr{X}}}_{k}\rangle}{\big\|\widetilde{\boldsymbol{\mathscr{X}}}_{k}\big\|_F} \widetilde{\boldsymbol{\mathscr{X}}}_{k}.
\end{align*}

Then, there exists some $\{b_r\}_{r=1}^{R}\subset\mathbb{C}$ such that
\begin{align*}
&\boldsymbol{\mathscr{X}} = \sum_{r=1}^{R}b_r\widetilde{\boldsymbol{\mathscr{X}}}_{r} \\ &= \sum_{r=1}^{R-1}b_r\widetilde{\boldsymbol{\mathscr{X}}}_{r} + b_R\left(\boldsymbol{\mathscr{X}}_R - \sum_{k=1}^{R-1}\frac{\langle\boldsymbol{\mathscr{X}}_R, \widetilde{\boldsymbol{\mathscr{X}}}_{k}\rangle}{\big\|\widetilde{\boldsymbol{\mathscr{X}}}_{k}\big\|_F} \widetilde{\boldsymbol{\mathscr{X}}}_{k}\right) \\
&= \sum_{r=1}^{R-1}c_r\widetilde{\boldsymbol{\mathscr{X}}}_{r} + b_R\boldsymbol{\mathscr{X}}_R.
\end{align*}
The second equality holds by using the definition of $\widetilde{\boldsymbol{\mathscr{X}}}_R$ in the Gram-Schmidt process and the last equality defines $\{c_r\}_{r=1}^{R-1}\subset\mathbb{C}$ by absorbing all the respective coefficients of $\{\widetilde{\boldsymbol{\mathscr{X}}}_{r}\}_{r=1}^{R-1}$.

By the Gram-Schmidt process, the span$\{\widetilde{\boldsymbol{\mathscr{X}}}_{r}\}_{r=1}^{R-1} =$ span$\{\boldsymbol{\mathscr{X}}_{r}\}_{r=1}^{R-1}$. Therefore, there exists coefficients $\{d_r\}_{r=1}^{R-1}\subset\mathbb{C}$ such that
\[
\boldsymbol{\mathscr{X}} = \sum_{r=1}^{R-1}d_r\boldsymbol{\mathscr{X}}_{r} + b_R\boldsymbol{\mathscr{X}}_R.
\]
Comparing this expansion to \eqref{CPX}, by linear independence we obtain that $b_R = a_R$ and $d_r = a_r$ for each $r\in \{1,\cdots,R-1\}$. We have shown
\[
\max_{r}|a_r| = |a_R| = |b_R| = \Bigg|\Bigg\langle
\boldsymbol{\mathscr{X}},\frac{\widetilde{\boldsymbol{\mathscr{X}}}_{R}}{\|\widetilde{\boldsymbol{\mathscr{X}}}_{R}\|_F}\Bigg\rangle\Bigg| \leq \|\boldsymbol{\mathscr{X}}\|_F.
\]

Our previous observations give that
\[
\|\boldsymbol{\mathscr{X}}\|_{*} \leq \sum_{r=1}^{R}|a_r| \leq R\|\boldsymbol{\mathscr{X}}\|_F .
\]
    
\end{proof}

\subsubsection{Proof of Lemma \ref{lem:minimaxpacking}}
\begin{proof}[Proof of Lemma \ref{lem:minimaxpacking}]
    We adapt the proof of \cite[Lemma~9]{lee2020tensor} and \cite[Lemma~8]{minimaxbin}. We first define the following set of matrices:
    $$\mathcal{C} = \left\{ \bM = (m_{ij}) \in \mathbb{R}^{I \times R} : m_{ij} \in \left\{ \tilde{\beta}, \tilde{\beta} + \epsilon (\tilde{\alpha} - \tilde{\beta}) \right\}, 1 \leq i \leq I, 1 \leq j \leq R\right\}.$$
    We next construct a set of block tensors. 
    First, for fixed $\bM \in \mathcal{C}$, define the matrix
    $$\bA(\bM) = \begin{bmatrix}
        \bM & \bM & \ldots \bM & \bM[:, 1:c_M]
    \end{bmatrix} \in \mathbb{R}^{I \times I},$$
    where the matrix $\bM$ repeats $\lfloor I / R \rfloor$ times and then add the first $c_M = (I - R \lfloor I / R \rfloor)$ columns of $\bM$.
    Here, $\lfloor \cdot \rfloor$ denotes the floor function (integer part).
    Then, consider the set 
    $$\mathcal{B} = \mathcal{B}(\mathcal{C}) = \left\{\bA(\bM) \circ \b1 \circ \b1 \circ \cdots \b1 : \bM \in \mathcal{C} \right\},$$
    where $\b1 \in \mathbb{R}^{I}$ is a vector of all $1$s. 

    By construction, every element in $\mathcal{C}$ has rank at most $R$ and hence the matrices $\bA$ also have rank at most $R$. 
    It follows that every tensor in $\mathcal{B}$ has CP rank at most $R$. 
    Moreover, the entries of tensors in $\mathcal{B}$ are either $\tilde{\beta}$ or $\tilde{\beta} + \epsilon (\tilde{\alpha} - \tilde{\beta})$.
    Noting that $0 < \tilde{\beta} < \tilde{\alpha}$, we have that 
    $$\tilde{\beta} + \epsilon (\tilde{\alpha} - \tilde{\beta}) \geq \tilde{\beta}.$$
    We also see that 
    $$\epsilon \tilde{\alpha} + (1 - \epsilon) \tilde{\beta} < \epsilon \tilde{\alpha} + (1 - \epsilon) \tilde{\alpha} = \tilde{\alpha},$$
    so that the entries of elements of $\mathcal{B}$ are bounded below by $\tilde{\beta}$ and above by $\tilde{\alpha}$ and $\mathcal{B} \subset \mathcal{S}_{R}(\tilde{\beta}, \tilde{\alpha})$. 
    
    We now note that the set $\mathcal{C}$ can be identified with the set of binary vectors $\{0, 1\}^{I R}$ after vectorization, subtracting $\tilde{\beta}$, and dividing by $\epsilon (\tilde{\alpha} - \tilde{\beta})$. 
    Since the transformed $\mathcal{C}$ contains the zero matrix, we may then apply Lemma \ref{lem:vg} and conclude that there is a subset $\widetilde{\mathcal{C}}$ of $\mathcal{C}$ such that after transformation 
    $\widetilde{\mathcal{C}}$ contains the zero matrix, $\left|\widetilde{\mathcal{C}}\right| \geq 2^{I R / 8}$, and for distinct elements $\bM^{(i)}, \bM^{(j)} \in \widetilde{\mathcal{C}}$, 
    $$\left\|\bM^{(i)} - \bM^{(j)}\right\|_0 \geq  \frac{I R}{8}.$$
    That is, the two matrices differ in at least $\frac{I R}{8}$ elements since $\|\cdot\|_0$ provides the number of non-zero entries. Noting that by construction, if two corresponding elements have a difference, their difference must be equal to $\epsilon (\tilde{\alpha} - \tilde{\beta})$, we conclude that 
\begin{equation}\label{eq:matrix_packing_spacing}
         \left\|\bM^{(i)} - \bM^{(j)}\right\|_F \geq \epsilon (\tilde{\alpha} - \tilde{\beta}) \sqrt{\frac{I_1 R}{8}}.
    \end{equation}
    
    As there is a one-to-one correspondence between $\mathcal{B}$ and $\mathcal{C}$, we may  identify the subset $\mathcal{F} \subseteq \mathcal{B}$ that corresponds to  $\widetilde{\mathcal{C}} \subseteq \mathcal{C}$.
    This subset $\mathcal{F}$  has the same cardinality as $\widetilde{\mathcal{C}}$.

    It remains to bound the difference between distinct elements of $\mathcal{F}$. 
    By construction, the elements of $\mathcal{F}$ have constant values when fixing all but the first two indices.
    It follows that for distinct elements $\boldsymbol{\mathscr{M}}^{(i)}, \boldsymbol{\mathscr{M}}^{(j)} \in \mathcal{F}$,
    \begin{align*}
        \left\| \boldsymbol{\mathscr{M}}^{(i)} - \boldsymbol{\mathscr{M}}^{(j)} \right\|_F^2 &\geq \frac{I R}{8} \epsilon^2 (\tilde{\alpha} - \tilde{\beta})^2 \left\lfloor \frac{I}{R}\right\rfloor \prod_{i = 3}^N I_i \\
        &\geq \frac{I R}{8} \epsilon^2 (\tilde{\alpha} - \tilde{\beta})^2  \frac{I}{ 2 R } \prod_{i = 3}^N I_i \\ 
        &= \frac{I^N}{16} \epsilon^2 (\tilde{\alpha} - \tilde{\beta})^2,
    \end{align*}
    where the first inequality follows by combining \eqref{eq:matrix_packing_spacing} with the structure of the tensors, the second by noting that $\left\lfloor I / R\right\rfloor \geq I / (2 R)$ since $\lfloor x \rfloor / x \geq 1/2$ for $x \geq 1$ (and $I / R \geq 1$), and the third by simplification. 

    To bound the difference from above, we note that the difference in any corresponding entries of $\boldsymbol{\mathscr{M}}^{(i)}$ and $\boldsymbol{\mathscr{M}}^{(j)}$ is either $0$ or $\epsilon (\tilde{\alpha} - \tilde{\beta})$. 
    There are $I_T$ total entries in the difference tensor, so that we immediately conclude that 
    $$\left\|\boldsymbol{\mathscr{M}}^{(i)} - \boldsymbol{\mathscr{M}}^{(j)} \right\|_F^2 \leq \epsilon^2 (\tilde{\alpha} - \tilde{\beta})^2 I^N.$$
\end{proof}

\section{Conclusions and Future Work}
\label{conclusion}

This study provides a theoretical foundation for assessing the reliability of tensor CP-based multiway analysis through a non-asymptotic interpretation of the Cram\'{e}r--Rao lower bound. By focusing on a rank-one Poisson tensor model, we demonstrate that constrained maximum likelihood estimators can achieve near-optimal decomposition factor estimation, approaching the CRLB up to constant and logarithmic factors. Our relaxed CRLB framework thus serves as a valuable benchmarking tool for evaluating estimator performance in structured, high-dimensional settings.

For higher-rank CP decompositions, the limitations of current estimation methods become more apparent. Our numerical experiments reveal a growing discrepancy between estimator variance and the CRLB as tensor rank increases. This divergence underscores the impact of degeneracy and non-identifiability that plagues high-rank tensor decompositions. Nonetheless, we find that CP-based inference retains near-minimax optimality in estimating the full parameter tensor, with only modest rank-dependent suboptimality.

Future work could extend this framework to broader classes of tensor models and noise distributions—such as Gaussian or overdispersed count data—to evaluate the generality of the near-efficiency results demonstrated for the Poisson rank-one case. Another important direction is the development of computationally tractable algorithms that approach the CRLB in higher-rank settings, where achieving near-efficiency remains a critical open challenge. In this work, we applied our recent results from \cite{PCPFIMpaper}, which provide closed-form expressions for the Fisher Information Matrix (FIM) for arbitrary tensor ranks. Further exploration of these explicit FIMs could reveal how to structure CP estimators to yield unique minimum variance solutions and ensure the FIM’s non-singularity. Understanding the null-space properties of these FIMs may ultimately offer a principled path toward designing more efficient tensor decomposition estimators.

\bibliographystyle{unsrt}
\bibliography{references}

\end{document}